\documentclass[12pt, notitlepage]{article}
\usepackage[utf8]{inputenc}
\usepackage[english]{babel}
\usepackage{amsfonts}
\usepackage{amssymb}
\usepackage{amsmath}
\usepackage{amsthm}
\usepackage{tocloft}
\usepackage{CJK}
\usepackage{tikz-cd}
\usepackage{mathrsfs}
\usepackage[toc,page]{appendix}
\newcommand{\vin}{\rotatebox[origin=c]{-90}{$\in$}}

\usepackage[english]{babel}

\usepackage{bm} 
\usepackage{chngcntr}
\makeatletter
\def\blfootnote{\xdef\@thefnmark{}\@footnotetext}
\makeatother

\counterwithin{equation}{section}

\makeatletter\@addtoreset{chapter}{part}\makeatother

\newcommand{\xdownarrow}[1]{%
  {\left\downarrow\vbox to #1{}\right.\kern-\nulldelimiterspace}
}

\begin{document}

\title{Real intersection theory (II)}

 \author{B. Wang\\
\begin{CJK}{UTF8}{gbsn}
(汪      镔)
\end{CJK}}

\maketitle

\newcommand{\hookuparrow}{\mathrel{\rotatebox[origin=c]{90}{$\hookrightarrow$}}}
\newcommand{\hookdownarrow}{\mathrel{\rotatebox[origin=c]{-90}{$\hookrightarrow$}}}
\newcommand\ddaaux{\rotatebox[origin=c]{-90}{\scalebox{0.70}{$\dashrightarrow$}}} 
\newcommand\dashdownarrow{\mathrel{\text{\ddaaux}}}

\newtheorem{thm}{Theorem}[section]

\newtheorem{ass}[thm]{\bf {Claim} }
\newtheorem{prop}[thm]{\bf {Property } }
\newtheorem{prodef}[thm]{\bf {Proposition and Definition } }
\newtheorem{construction}[thm]{\bf {Main construction } }
\newtheorem{assumption}[thm]{\bf {Assumption} }
\newtheorem{proposition}[thm]{\bf {Proposition} }
\newtheorem{theorem}[thm]{\bf {Theorem} }
\newtheorem{apd}[thm]{\bf {Algebraic Poincar\'e duality} }
\newtheorem{cond}[thm]{\bf {Condition} }
\newtheorem{ex}[thm]{\bf Example}
\newtheorem{corollary}[thm]{\bf Corollary}
\newtheorem{definition}[thm]{\bf Definition}
\newtheorem{lemma}[thm]{\bf Lemma}
\newtheorem{con}[thm]{C}
\newtheorem{conj}[thm]{\bf Conjecture}

\begin{abstract}
Continuing from part (I), we develop properties of real
 intersection theory that turns out  to be an
extension of the well-established theory  in algebraic geometry. 
	\end{abstract}

\bigskip


\tableofcontents

\section{Introduction}  Intersection   in mathematics has a long history. 
But the systemically developed  theories only started appearing in the 20th century. 
They are centered around the  quotient rings that are derived  from  freely generated Abelian groups of  non-quotient objects.
  While the  quotient rings  fit into the known axiomatic system well,
 the  non-quotient groups do not. \bigskip

For instance  in classical approach  the topological intersection for a real compact manifold $X$ in homology or/and  
cohomology (quotient groups) are obtained
from the non quotient objects -- singular chains.  Once the theory is set up,  quotient groups are well-adapted to the  axiomatic
environment, but the groups of singular chains are not.  Using cohomology
we can let $H^i(X;\mathbb Z)$ be the cohomology of degree $i$ with integer coefficients. The product is defined in an
elaborated method through the intersection of chains as the cup product, 
\begin{equation}\begin{array}{ccc}
\cup: H^p(X;\mathbb Q)\times H^q(X;\mathbb Z) &\rightarrow  & H^{p+q}(X;\mathbb Z). \end{array}
\end{equation}
The  product gives a ring structure to the cohomology.  Then the focus is shifted to the framework of rings which provides rich
mathematical structures. 
 However 
such a ring structure does not exist on the groups of chains. So the groups of singular chains plays a supporting role behind
the cohomology groups. 
   \bigskip

Another example is the intersection in algebraic geometry defined by William Fulton in [3].  Its non-quotient  objects are 
algebraic cycles, which form an Abelian group $Z$.  
Let $X$ be a smooth projective variety over a field.  Let $CH(X)$ be the Chow groups (quotient groups). Then
there is an elaborated product map through the intersection of algebraic cycles,
 \begin{equation}\begin{array}{ccc}
\bullet : CH(X)\times CH(X) &\rightarrow  & CH(X) \end{array}
\end{equation}
such that $CH(X)$ is a ring.  The study of algebraic cycles has been motivated by the structure of the Chow rings. 
So the non-quotient $Z$ group plays a supporting role  since it does not have the similar ring structure.
\bigskip

Both ring structures $\cup, \bullet$ are classically known to be compatible with each other and functorial 
 between  spaces. These quotioned intersections  play important roles
in the more general cohomological theory that later generalizes both quotient rings
in the derive category. They created  rich algebraic structures that flourish beyond the original transcendental geometry.
   However the flourishing  has been one-sided 
focusing on quotient groups mainly through homological algebra.  
In this paper  
we try to look into the other direction which studies the non-quotient objects. 
Our direction shows  that the global invariant -- cohomology ring may 
 contain  non-trivial geometry without quotient in the local charts. More specifically, the intersection in cohomology
which is considered to be a quotient object, is related to  
the geometric measure in transcendental geometry, which is a non-quotient object.  
 The study of non-quotient objects for quotient objects  is the real intersection theory. 
\bigskip

On the technical side, we let $\mathcal X$ be a connected, oriented manifold of dimension $m$. In [9], using local charts, we defined 
the subspace  $\mathcal C(\mathcal X)$ called Lebesgue currents in the space of currents  such that there is a bilinear homomorphism -- the intersection of  currents, 
\begin{equation}\begin{array}{ccc}
[\cdot\wedge \cdot]: \mathcal C(\mathcal X)\times \mathcal C(\mathcal X) &\rightarrow & \mathcal C(\mathcal X)\\
(T_1, T_2) &\rightarrow & [T_1\wedge T_2].
\end{array}\end{equation}
The intersection is extrinsic,  for it depends on De Rham data consisting of a De Rham covering of $\mathcal X$.  But it satisfies
an important intrinsic relation,
\begin{equation}
supp([T_1\wedge T_2])\subset supp(T_1)\cap supp(T_2).
\end{equation}
where the support is the focus of the theory. 
In this paper which is not self-contained, we continue [9] to establish sufficient 
 properties for  the applications.  It is a detailed demonstration that by adjusting the extrinsic data, 
the well-known operations on quotient objects can be
carried over to non-quotient objects \footnote{Intrinsically it is impossible.}.  As a conclusion  we give a short but immediate application that 
proves  the generalized Hodge conjecture on 3-folds.

\section{Basic properties}
\bigskip

Let $\mathcal X$ be a connected oriented manifold.  Let $T_1, T_2$ be two arbitrary singular chains. Then
$T_1\cap T_2$ may not be a singular chain.  To catch the essence of such an intersection we are
 going to use a tool --- current,  the notion  created by G. de Rham in 1950's, [2]. \bigskip

We denote the space of $C^\infty$ forms with a compact support by $\mathscr D(\mathcal X)$, and its topological dual -- the space of currents  by $\mathscr D'(\mathcal X)$.  If necessary we add a superscript(a subscript)
to denote the codimension (the dimension) of currents.  The functional of a current $T$ is denoted by
\begin{equation}
\int_T (\bullet).
\end{equation}

 \bigskip

\begin{lemma}
Let $\mathcal Z\subset \mathcal X$ be a submanifold. Let 
$$\begin{array}{ccc}\mathcal Z &\stackrel{i}\hookrightarrow & \mathcal X\end{array}$$ be the inclusion map. 
Let 
\begin{equation}
\mathscr D(\mathcal X, \mathcal Z)=\{\phi\in \mathscr D(\mathcal X): \phi|_{\mathcal Z}=0\},
\end{equation}
where $\phi|_\mathcal Z$ is the pullback of the $C^\infty$-differential form by the inclusion map.
So \begin{equation}
\mathscr D(\mathcal X, \mathcal Z)\subset \mathscr D(\mathcal X).
\end{equation}

Then the  sequence of topological dual
\begin{equation}\begin{array}{ccccc}
\mathscr D'(\mathcal Z) &\stackrel{i_\ast}\rightarrow & \mathscr D'(\mathcal X) &\stackrel{R}\rightarrow &
 \mathscr D'(\mathcal X, \mathcal Z)
\end{array}\end{equation}
is exact, where $'$ stands for the topological dual and $R$ is the restriction map. 

\end{lemma}

\bigskip

\begin{proof} We may assume $\mathcal Z$ is compact. It is trivial that $ R\circ i_\ast=0$. Let's show 
$$ker (R)\subset Im(i_\ast).$$

Let $U$ be a tubular neighborhood of $\mathcal Z$ and $j: U\to \mathcal Z$ be a
 projection induced from the normal bundle structure of $U$.
Let $h$ be a $C^\infty$ function on $\mathcal X$ such that it has a compact support in $U$ and 
it is $1$ on $\mathcal Z$. 
  For any $T\in \mathscr D'(\mathcal X)$, we define a current $T'$ on $\mathcal Z$

\begin{equation}
\int_{T'} (\cdot)= \int_{T} h j^{\ast}(\cdot) .\end{equation}
Let $T\in ker(R)$. We would like to show 
$$i_\ast (T')=T.$$
It suffices to show that for  any testing form of $\phi$ on $\mathcal X$
$$\int_{T} hj^{\ast} (\phi|_\mathcal Z)=\int_T \phi,$$
or
\begin{equation}
\int_T  \biggl(h j^{\ast} (\phi|_\mathcal Z)-\phi\biggr)=0
.\end{equation}
Since $h j^{\ast} (\phi|_\mathcal Z)-\phi$ vanishes on $\mathcal Z$, the formula (2.6) holds.  
If $\mathcal Z$ is non-compact, we can use a partition of
unity to have the same proof.
We complete the proof.

\end{proof}

\bigskip

In [9], we have developed the local and global calculations of intersection of Lebesgue currents. 
We refer the readers to [2] and [9] for the necessary background material.
Let's give a summary. It started with the De Rham's regularization.

\begin{theorem} (G. de Rham)\quad\par
Let $\mathcal X$ be a connected, oriented manifold. 
Let $\epsilon$ be a small positive number.  There are linear operators $R_\epsilon$ and 
  $A_\epsilon$ on $\mathscr D'(\mathcal X)$ satisfying
\par
(1) a homotopy formula
\begin{equation}
R_\epsilon T-T=b A_\epsilon T+ A_\epsilon bT.
\end{equation}
\par \hspace{1cc} where $b$ is the boundary operator. 
\par
(2) $supp(R_\epsilon T), supp(A_\epsilon T) $ are contained in any given neighborhood of\par \hspace{1cc}
$supp(T)$ provided $\epsilon$ is sufficiently small.\par
 
(3) $R_\epsilon T$ is $C^\infty$;\par
(4)  If $T$ is $C^r$, $A_\epsilon T$ is $C^r$.\par
(5) If a smooth differential form $\phi$ varies in a bounded set and $\epsilon$ is bounded \par \hspace{1cc} above, then
$R_\epsilon \phi, A_\epsilon \phi$ are bounded.\par
(6) As $\epsilon\to 0$, 
$$R_\epsilon T(\phi)\to T(\phi), A_\epsilon T(\phi)\to 0$$
 \par \hspace{1cc} uniformly on each bounded set $\phi$. 
\bigskip

In this paper we use the notations $R_\epsilon^{\mathcal X}, A_\epsilon^{\mathcal X}$ to replace
De Rham's notations $R_\epsilon, A_\epsilon$. 
\end{theorem}
\bigskip

\begin{definition} (for De Rham's regularization)\quad\par
(a) We call $R_\epsilon^\mathcal X$ from Theorem 2.2 the De Rham's smoothing operator,  \par\hspace{1cc} 
$ A_\epsilon^\mathcal X$ from Theorem 2.2 the De Rham's homotopy operator,   
 and the \par\hspace{1cc} 
associated regularization the De Rham's   regularization. \par
(b)    
We define De Rham data to be all items in the construction 

\end{definition}
 
\bigskip
We continue to have

\bigskip

\begin{theorem} (B. Wang 2019)\quad\par
 Let $\mathcal X$ be a connected, oriented manifold of dimension $m$, equipped with a De Rham data.  
Let $T_1, T_2\in \mathcal C(\mathcal X)$ such that $dim(T_1)+dim(T_2)\geq m$.  \par
(1) Then there is a family of $C^\infty$ forms, the kernel  
$\varrho_\epsilon(\mathbf x, \mathbf y)$ on $\mathcal X\times \mathcal X$ where $\epsilon>0$ such that
\begin{equation}
R_\epsilon^{\mathcal X} (T_2) =\int_{ \mathbf y\in T_2} \varrho_\epsilon(\mathbf x, \mathbf y)
\end{equation}
where the right hand side is defined by the fibre integral similar to the fibre integral of the projection 
$$\begin{array}{ccc} \mathcal X\times \mathcal X &\rightarrow & \mathcal X (1st\ copy)\\
(\mathbf x, \mathbf y)&\rightarrow & \mathbf x.\end{array}$$
Furthermore   $\varrho_\epsilon(\mathbf x, \mathbf y)$ is a closed form,  which in case with a compact $\mathcal X$
 is Poincar\'e dual  to the diagonal of $\mathcal X\times \mathcal X$ in the cohomology group.\footnote{In the language of [2], 
 $\varrho_\epsilon(\mathbf x, \mathbf y)$   is homologous to the diagonal for each non-zero $\epsilon$.  }

\par
(2) There exists a subspace $\mathcal C(\mathcal X)$ of $\mathscr D'(\mathcal X)$, defined in geometric measure theory, such that
for $T_1, T_2\in \mathcal C(\mathcal X)$ continuous functional
\begin{equation}\begin{array}{ccc}
\phi &\rightarrow & \displaystyle{\lim_{\epsilon\to 0}}\int_{T_1}R_\epsilon^{\mathcal X} (T_2)\wedge \phi\\
\vin & &\vin\\
\mathscr D(\mathcal X) && \mathbb R\\
\end{array}
\end{equation}
exists and lies in $ \mathcal C(\mathcal X)$, denoted by
\begin{equation}
[T_1\wedge T_2].
\end{equation}

\end{theorem}

\bigskip

\begin{proposition} Let $\mathcal X$ be a manifold equipped with a De Rham data.
Let $i: \mathcal Z\hookrightarrow \mathcal X$ be a submanifold.
Let $T\in \mathcal C(\mathcal X)$. Then there is a current denoted
by $[\mathcal Z\wedge T]_{\mathcal Z}$ in $\mathcal Z$ such that
\begin{equation}
i_\ast ([\mathcal Z\wedge T]_{\mathcal Z})=[\mathcal Z\wedge T].
\end{equation}

\end{proposition}
\bigskip

\begin{proof}

For any $\phi\in \mathscr D(\mathcal X, \mathcal Z)$, 
\begin{equation}
\int_{[\mathcal Z\wedge T]} \phi
=\displaystyle{\lim_{\epsilon\to 0}}
\int_{\mathcal Z}R_\epsilon^\mathcal X(T)\wedge\phi=0.
\end{equation}
Then by Lemma 2.1, there is a current in $\mathcal Z$ satisfying (2.11).

\end{proof}
\bigskip

\subsection{Basic properties}

\begin{prop}\quad\par

Let $\mathcal X$ a connected, oriented $C^\infty$ manifold of dimension $m$. Assume it is equipped with a De Rham data. 
In [9], we defined the intersection of  currents,
\begin{equation}\begin{array}{ccc}
\mathcal C(\mathcal X)\times \mathcal C(\mathcal X) &\rightarrow &  \mathcal C(\mathcal X) \\
(T_1, T_2) &\rightarrow & [T_1\wedge T_2],
\end{array}\end{equation}
on the subspace  $\mathcal C(\mathcal X)\subset \mathscr D'(\mathcal X)$ consisting  of  Lebesgue currents.
For all Lebesgue currents $T_1, T_2$, the intersection $[T_1\wedge T_2]$ has properties:
\bigskip

(1) (Supportivity) 
\begin{equation}
supp([T_1\wedge T_2])\subset supp(T_1)\cap supp(T_2). \end{equation}

(2) (Closedness) The intersection current $[T_1\wedge T_2] $ is  closed if $T_1, T_2$ are.\par

\bigskip

(3) (Graded commutativity)   There is a graded-commutativity. If 
$$deg(T_1)=p, deg(T_2)=q, $$ then 
\begin{equation}
[T_1\wedge T_2]=(-1)^{pq} [T_2\wedge T_1]
\end{equation}

or equivalently

\begin{equation}
\displaystyle{\lim_{\epsilon'\to 0}} \displaystyle{\lim_{\epsilon\to 0}}\int_{\mathcal X} R_{\epsilon}(T_1)
\wedge R_{\epsilon'}T_2\wedge \phi
=\displaystyle{\lim_{\epsilon\to 0}}
\displaystyle{\lim_{\epsilon'\to 0}}\int_{\mathcal X} 
R_{\epsilon}(T_1)\wedge R_{\epsilon'}T_2\wedge \phi.
\end{equation}
 for the test from $\phi\in \mathscr D(\mathcal X)$. 
\par

\bigskip

(4) (Cohomologicity) Let $\mathcal X$ be compact. We use $\langle T\rangle $ to denote the cohomology class represented by a closed current $T$.
If  $T_1, T_2$ are closed, in $H(\mathcal X;\mathbb R)$,  we have 
\begin{equation}
 \langle T_1\rangle\cup \langle T_2\rangle=\langle [T_1\wedge T_2]\rangle.
\end{equation}\par
Hence if  the cohomology  $\langle T_1\rangle,  \langle T_2\rangle $ are integral, so is $\langle [T_1\wedge T_2]\rangle$.

(5) (Associativity)  There is an associativity
\begin{equation}
\biggl[[T_1\wedge T_2]\wedge T_3\biggr]=\biggl[T_1\wedge [T_2\wedge T_3]\biggr].\end{equation}

\end{prop}

(6) (Leibniz rule) If $bT_1, bT_2$ are Lebesgue and $deg(T_1)=p,$ 
then
\begin{equation}
d[T_1\wedge T_2]=[dT_1\wedge T_2]+(-1)^{p} [T_1\wedge dT_2].
\end{equation}

\bigskip

\begin{proof}\quad \par

(1)  The proof is in  [9].

\bigskip

(2) Let $\phi$ be a test form.
By the definition 
\begin{align}\begin{split}
 &\int_{ b [T_1\wedge T_2]} \phi
\\ &=\lim_{\epsilon\to 0} \int_{T_1}  R_\epsilon T_2\wedge d\phi\\
 &=\pm \int_{T_1} d R_\epsilon T_2\wedge \phi\end{split}
\end{align}

According to the homotopy (2.7),  
\begin{equation}
b R_\epsilon T_2-b T_2= bb A_\epsilon T_2- b A_\epsilon b T_2 
\end{equation}
  Because $T_2 $ is closed, 
$$b R_\epsilon T_2=0.$$
So  $ [T_1\wedge T_2]$ is closed.

\bigskip

\bigskip

(3) (Graded commutativity ).  
The proof is in  [9].

\bigskip

(4)  Let $\phi$ be a closed $C^\infty$ form of degree $deg(T_1)+deg(T_2)$. 
Denote the intersection number in cohomology by $int(\cdot, \cdot)$.
Hence the intersection number,
\begin{equation}
 int(\langle [T_1\wedge T_2]\rangle, \langle \phi\rangle) \end{equation}
is a well-defined real number that equals to 
\begin{equation} 
\displaystyle{\lim_{\epsilon\to 0}}\int_{T_1} R_{\epsilon}(T_2)\wedge \phi.\end{equation}
(by De Rham's theorem) which is \begin{equation} 
(-1)^{deg^2(T_1)}\displaystyle{\lim_{ \epsilon\to 0 }}\int_{T_2\wedge \phi} R_{\epsilon}(T_1)\\
\end{equation}
(by the community (3)). 
Then we use (2.24) as  De Rham's Kronecker index
$$T_1\wedge (T_2\wedge \phi)[1]$$ which is well-defined between the currents $T_1$ and $T_2\wedge \phi$.
We obtain that 
\begin{equation}int(\langle [T_1\wedge T_2]\rangle, \langle \phi\rangle)=
T_1\wedge (T_2\wedge \phi) [1]\end{equation}
 
On the compact manifold, the Kronecker index only depends the cohomology classes (see section 20, chapter IV [2]). 
Hence (2.25) is the intersection number  
\begin{equation}
int( \langle T_1\rangle, \langle T_2\wedge \phi\rangle)
\end{equation}
which by the associtivity of the intersection in cohomology is the same as 
\begin{equation}
int( \langle T_1\rangle \cup  \langle T_2\rangle,   \langle\phi\rangle)
\end{equation}
Since $\phi$ is any closed form, 
\begin{equation}
\langle T_1\rangle \cup  \langle T_2\rangle=\langle [T_1\wedge T_2]\rangle.
\end{equation}
\bigskip

(5) Let degrees of $T_i$ be $p_i$. The intersection of Lebesgue currents is still Lebesgue. Therefore we have a triple intersection expressed as a functional  on the test forms $\bullet $.

Then
\begin{equation}\begin{array}{c}
\int_{\biggl[[T_1\wedge T_2]\wedge T_3\biggr]} (\bullet) \\
=\displaystyle{\lim_{\epsilon_3\to 0}}\displaystyle{\lim_{\epsilon_2\to 0}}\displaystyle{\lim_{\epsilon_1\to 0}}\int_{\mathcal X} R_{\epsilon_1}(T_1)\wedge R_{\epsilon_2}(T_2) 
\wedge R_{\epsilon_3}(T_3)\wedge (\bullet).\end{array}\end{equation}
By the commutitavity in  Proposition 4.4 of [9], 
\begin{equation}\begin{array}{c}
\int_{\biggl[T_1\wedge [T_2\wedge T_3]\biggr]} (\bullet) =\int_{(-1)^{p_1(p_2+p_3)}\biggl[ [T_2\wedge T_3]\wedge T_1\biggr]} (\bullet)\\
=(-1)^{p_1(p_2+p_3)}\displaystyle{\lim_{\epsilon_1\to 0}}\displaystyle{\lim_{\epsilon_3\to 0}}
\displaystyle{\lim_{\epsilon_2\to 0}\int_{\mathcal X}} R_{\epsilon_2}(T_2)\wedge R_{\epsilon_3}(T_3)\wedge R_{\epsilon_1}(T_1)
\wedge (\bullet) \\
=\displaystyle{\lim_{\epsilon_3\to 0}}\displaystyle{\lim_{\epsilon_2\to 0}}
\displaystyle{\lim_{\epsilon_1\to 0}}\int_{\mathcal X} R_{\epsilon_1}(T_1)\wedge R_{\epsilon_2}(T_2)\wedge R_{\epsilon_3}(T_3)\wedge (\bullet)\\
=\int_{\biggl[[T_1\wedge T_2]\wedge T_3\biggr]}(\bullet)
.\end{array}\end{equation}

\bigskip

(6) (Leibniz Rule) Let $\phi\in \mathscr D(\mathcal X)$ be a test form. 
Let $$deg(T_1)=p, deg(T_2)=q.$$
Then
\begin{align*}
 & b[T_1\wedge T_2](\phi)
\\ &= 
\displaystyle{\lim_{\epsilon\to 0}}\int_{T_1} R_\epsilon T_2\wedge d\phi\\
&=\displaystyle{\lim_{\epsilon\to 0}}\int_{T_1}\biggl( (-1)^{q} d (R_\epsilon T_2\wedge \phi)+(-1)^{q+1} dR_\epsilon T_2\wedge\phi\biggr) 
\\
&=\displaystyle{\lim_{\epsilon\to 0} }\int_{(-1)^{q} bT_1}R_\epsilon T_2\wedge\phi+
\displaystyle{\lim_{\epsilon\to 0}}\int_{(-1)^{q+1}T_1}dR_\epsilon T_2\wedge\phi\\
 &(\text{ $bT_1, bT_2$ are Lebesgue})\\
&=\int_{(-1)^{q}[bT_1\wedge T_2]}\phi+\int_{(-1)^{q+1}[T_1\wedge dT_2]}\phi
\end{align*}
Hence 

\begin{equation}
b[T_1\wedge T_2]=(-1)^{q}[bT_1\wedge T_2]+(-1)^{q+1}[T_1\wedge dT_2].
\end{equation}
After change the sign, we found (2.31) is the same as 
(2.19). 

\end{proof}

\subsection { Product and Inclusion}

Intersection theory in algebraic geometry has functoriality. We study its extension to the real intersection theory. 
\bigskip

\begin{definition}\quad \par

(1) Let $\mathcal U_1$, $\mathcal U_2$ be the De Rham data for the manifolds
$\mathcal X_1, \mathcal X_2$ respectively. Let $B_i^1\subset U_i, B_j^2\subset U_j^2$ be the De Rham covering from $\mathcal U_1, \mathcal U_2$  and $f^1_i, f^2_j$
are the associated convolution functions.  
 Then  given an order of  the  De Rham's covering $B_i^1\times B_j^2 \subset U_i^1\times U_j^2$ 
of $\mathcal X_1\times \mathcal X_2$, we define the De Rham data on the product $\mathcal X_1\times \mathcal X_2$ by taking the 
direct product of given data in each manifold. For instance,  
we define the convolution  functions $ (f_i^1, f_j^2)$  on each open set $B_i^1\times B_j^2$.  We call it the
product De Rham data.  \par
(2) Given $(B_1^i, B_2^j)\subset (U_1^i, U_2^j)$  the product  De Rham covering of $\mathcal X_1\times \mathcal X_2$ that 
has the given order denoted by
$$(U_1^{i_1}, U_2^{j_1}),  (U_1^{i_2}, U_2^{j_2}), \cdots, $$
we construct  new De Rham data on $\mathcal X_1$
by setting ordered De Rham covering as $$U_1^{i_1}, U_1^{i_2}, U_1^{i_3}, \cdots $$  and  on $\mathcal X_2$ by setting 
the De Rham covering as  
$$U_2^{j_1}, U_2^{j_2}, U_2^{j_3}, \cdots$$
These new De Rham data are called projection De Rham data
with respect to the product De Rham data.

\end{definition}

The important implication of these new De Rham data is the following projection formula.\bigskip

\begin{proposition}(Projection formula)
Let $\mathcal X_1\times \mathcal X_2$ be two  manifolds  equipped with a product De Rham data,
$X_2$ equipped with a projection De Rham data. Let  
 $P_2: \mathcal X_1\times \mathcal X_2\to  \mathcal X_2$ be the projection, and 
$\sigma\in \mathcal C(\mathcal X_2)$,  and $T\in \mathcal C(X_1\times X_2)$. Then\par

(1) \begin{equation}
R_\epsilon^{\mathcal X_1\times \mathcal X_2}(\mathcal X_1\otimes \sigma)=(P_2)^\ast (R_\epsilon^{\mathcal X_2} (\sigma)).
\end{equation}

(2) Let $\mathcal X_1$ be compact. Then

\begin{equation}
[(P_2)_\ast (T) \wedge \sigma]=(P_2)_\ast [T\wedge (\mathcal X_1\otimes \sigma)].
\end{equation}

\end{proposition}

\begin{proof}
(1). 
Assume $\mathcal X_1, \mathcal X_2$ are equipped with De Rham data, $\mathcal U_1, \mathcal U_2$ respectively.
Let's give a product De Rham data to $\mathcal X_1\times \mathcal X_2$ 
and projection De Rham data to $\mathcal X_1$ and $\mathcal X_2$.
Each chart in the product De Rham covering $(U_1^i, U_2^j)$ gives a De Rham's smoothing operators 
  on $\mathcal X_1\times \mathcal X_2$, denoted by $R_\epsilon^{\mathcal X_1\times \mathcal X_2, (i, j)}$. Each chart in  the projection 
De Rham covering $U_2^j$ gives a De Rham's smoothing operator  on $\mathcal X_2$, denoted by $R_\epsilon^{\mathcal X_2, j}$. 
We claim for any $\sigma\in \mathscr D'(\mathcal X)$, 
\begin{ass}
\begin{equation} R_\epsilon^{\mathcal X_1\times \mathcal X_2, (i, j)}(1\otimes \sigma)
=1\otimes R_\epsilon^{\mathcal X_2, j}(\sigma)\end{equation}
on $\mathcal X_1\times \mathcal X_2-\partial (B_i\times B_j)$.
\end{ass}
\noindent where $1$ is the current $\mathcal X_1$. 
Let's consider both sides of (2.34) restricted to a neighborhood of a point $(q_i, q_j)$ (both sides globally are not continuous, but in 
a neighborhood they are).  
If  $(q_i, q_j)\notin \bar B_1^i\times \bar B_2^j$. Then both sides are restricted to the $1\otimes \sigma$ in the neighborhood. 
 If $(q_i, q_j)\in  B_1^i\times  B_2^j$, let $\phi\in \mathscr D^{dim(\mathcal X_1)+dim(\sigma)}(\mathcal X_1\times \mathcal X_2)$ be supported in
the neighborhood.  
\begin{align*}
 &\int_{R_\epsilon^{\mathcal X_1\times \mathcal X_2, (i, j)}(1\otimes \sigma)}\phi\\
&=\int_{\mathbf x_1\in B_i} \int_{\mathbf x_2\in B_j} R_\epsilon^{\mathcal X_2, j}(\sigma)\wedge \phi(\mathbf x_1, \mathbf x_2)
\\ &=\int_{1\otimes R_\epsilon^{\mathcal X_2, j}(\sigma)}\phi
\end{align*}
( we should note that the partial degree of $ \phi(\mathbf x_1, \mathbf x_2)$ in $\mathbf x_1$ is maximal).
Then  we obtain that
(2.34) holds on the $$\mathcal X_1\times \mathcal X_2- \partial (B_i\times B_j).$$
 Continuing from Claim 2.9, we glue each piece $(U_1^i, U_2^j)$ by taking De Rham's    compositions of
smoothing operators on both sides of (2.34) to obtain that
\begin{equation}
 R_\epsilon^{\mathcal X_1\times \mathcal X_2}(1\otimes \sigma)
=1\otimes R_\epsilon^{\mathcal X_2}(\sigma)
\end{equation}
on $$\mathcal X_1\times \mathcal X_2-\cup_{i, j} \partial (B_i\times B_j),$$ 
where $R_\epsilon^{\mathcal X_2}$ is obtained from the projection De Rham data with respect to the product 
De Rham data.  
Now since both sides of (2.35) are $C^\infty$, by the continuity, (2.35) holds on 
the closure $\mathcal X_1\times \mathcal X_2$. 
This completes the proof of part (1). 

\par
(2).  Since $\mathcal X_1$ is compact, $P_2$ is proper. Then the pushforward $(P_2)_\ast$ of currents is well-defined. 
 Let $\phi$ be a test form on $\mathcal X_2$.  We use projection De Rham 
data on $\mathcal X_2$ and product De Rham data on $\mathcal X_1\times \mathcal X_2$ to find 
\begin{align*}
 &\int_{[(P_2)_\ast (T) \wedge \sigma]} \phi
=\displaystyle{\lim_{\epsilon\to 0}}\int_{  (P_2)_\ast T}  R_\epsilon^{\mathcal X_2}(\sigma)\wedge \phi\\
 &=\displaystyle{\lim_{\epsilon\to 0}}\int_T P_2^\ast ( R_\epsilon^{\mathcal X_2}(\sigma)\wedge \phi)\\
& =\displaystyle{\lim_{\epsilon\to 0}}\int_T  (1\otimes R_\epsilon^{\mathcal X_2}(\sigma))\wedge P_2^\ast (\phi)
\end{align*}
\begin{align*}
& (\text {Use part (1)})\\
&=\displaystyle{\lim_{\epsilon\to 0}}\int_T  R_\epsilon^{\mathcal X_1\times \mathcal X_2}
 (1\otimes \sigma)\wedge P_2^\ast (\phi)\\
&=\int_{[T\wedge (\mathcal X_1\times \sigma)]} P_2^\ast (\phi)
\\&=\int_{(P_2)_\ast[T\wedge (\mathcal X_1\times \sigma)]} \phi
\end{align*}

 This completes the proof.
\end{proof}

\bigskip

\begin{proposition} ( reduction to the diagonal)\quad 
Let $\mathcal X$ be a compact manifold. Let $T_1, T_2\in \mathcal C(\mathcal X)$. Let $\mathcal X$ be equipped with a De Rham data
$\mathcal U$. We give the product De Rham data to $\mathcal X\times \mathcal X$ and the associated  projection De Rham data to
each copy $\mathcal X$ \footnote{ Each copy $\mathcal X$ will have different De Rham data depending on the order of the product De Rham data.}.
With these De Rham data we have the reduction to diagonal,
\begin{equation}
[T_1\wedge T_2]=(P_2)_\ast [ (T_1\otimes T_2)\wedge \Delta]
\end{equation}
where $P_2: \mathcal X\times \mathcal X\to \mathcal X(2nd \ copy)$ is the projection, 
the left hand side of intersection occurs in the second copy of $\mathcal X$, and $\Delta$ is the diagonal.

\end{proposition}
\bigskip

\begin{proof} 
With induced  De Rham data there is  the projection formula (Proposition 2.8), 
\begin{equation}
[T_1\wedge T_2]=(P_2)_\ast [T_1'\wedge (\mathcal X\otimes T_2)]
\end{equation}
where $T_1'$ is a current in $\mathcal X\times \mathcal X$ such that 
$$(P_2)_\ast T_1'=T_1.$$
Now we claim 
\begin{ass}
 with  the projection De Rham data also on the first copy $\mathcal X$, 
\begin{align}
& (P_2)_\ast [(T_1\otimes \mathcal X)\wedge \Delta]=T_1\\
& (T_1\otimes \mathcal X)\wedge (\mathcal X\otimes T_2)=T_1\otimes T_2.
\end{align}
\end{ass}
For (2.38), we let $\phi\in \mathscr D(\mathcal X)$. 
Then

\begin{align*}
 &\int_{(P_2)_\ast [(T_1\otimes \mathcal X)\wedge \Delta]}\phi\\
&=\int_{ [(T_1\otimes \mathcal X)\wedge \Delta]}(P_2)^\ast (\phi)\\
&=\displaystyle{\lim_{\epsilon\to 0}} \int_{\Delta} R_\epsilon^{\mathcal X\times \mathcal X}
(T_1\otimes \mathcal X)\wedge (P_2)^\ast(\phi)\\
& (\text{ Use projection formula, Proposition 2.8} )\\
&=\displaystyle{\lim_{\epsilon\to 0}} \int_{\Delta} (P_1)^\ast (R_\epsilon^{\mathcal X}(T_1))\wedge (P_2)^\ast(\phi)\\
 &(\text{ where $P_1: \mathcal X\times \mathcal X\to \mathcal X(1st\ copy)$ is the projection then identify $\mathcal X\simeq \Delta$})\\
&=\displaystyle{\lim_{\epsilon\to 0}} \int_{\mathcal X} R_\epsilon^{\mathcal X}(T_1)\wedge \phi\\
&=\int_{T_1}\phi.
\end{align*}
Hence (2.38) is true. For (2.39), we let $\phi\in \mathscr D(\mathcal X)$.  Then by the projection formula
\begin{align}
 R_\epsilon^{\mathcal X\times \mathcal X} ( \mathcal X\times T_2)=R_{\mathcal X}(T_2)
\end{align}
Then
\begin{align}\begin{split}
& \int_{(T_1\otimes \mathcal X)\wedge (\mathcal X\otimes T_2)}\phi\\
&=\displaystyle{\lim_{\epsilon\to 0}} \int_{T_1\otimes \mathcal X} 
R_\epsilon^{\mathcal X\times \mathcal X}(1\otimes T_2)\wedge \phi \\
&=\displaystyle{\lim_{\epsilon\to 0}} \int_{T_1\otimes \mathcal X} 
1\otimes R_\epsilon^{ \mathcal X}(T_2)\wedge \phi\\
&=\displaystyle{\lim_{\epsilon\to 0}} \int_{\mathcal X} R_\epsilon^{ \mathcal X}(T_2)
\wedge \bigl((\int_{\mathbf x_1\in T_1} \phi(\mathbf x_1, \mathbf x_2)\bigr)\\
&=\int_{T_2} \int_{\mathbf x_1\in T_1} \phi(\mathbf x_1, \mathbf x_2)\\
&=\int_{T_1\otimes T_2}\phi. \end{split}\end{align}
Hence with projection De Rham data on both $\mathcal X$ and product De Rham data on $\mathcal X\times \mathcal X$,
Claim 2.11 is true.  Then in (2.37), we replace $T'$ by $(T_1\otimes \mathcal X)\wedge \Delta$ and apply the associativity and 
commutativity to have
\begin{align*}
 [T_1\wedge T_2] & =
(P_2)_\ast \biggl( [(T_1\otimes \mathcal X)\wedge \Delta\wedge (\mathcal X\otimes T_2)]\biggr)\\
&=(P_2)_\ast \biggl( [  (T_1\otimes \mathcal X)\wedge (\mathcal X\otimes T_2)\wedge\Delta]\biggr)\\
& (\text{ by (2.39)} )\\
&=(P_2)_\ast \biggl( [ (T_1\otimes T_2)\wedge \Delta ]\biggr).
\end{align*}

This completes the proof. 

\end{proof}

\bigskip

Let $i: \mathcal Z\hookrightarrow  \mathcal X$ be the inclusion of a  submanifold of dimension $n$.

\bigskip

\bigskip

\bigskip

\begin{proposition}(associativity)\quad \par
 There exist De Rham data $\mathcal U_{\mathcal Z}, \mathcal U_\mathcal X$ on $\mathcal Z, \mathcal X$ respectively  such that
for any cellular chain $\mathcal W\subset \mathcal Z$ and $\sigma\in \mathcal C(\mathcal X)$,
\begin{equation}
i_\ast ([\mathcal W\wedge [\sigma\wedge \mathcal Z]_{\mathcal Z}])=[i_\ast\mathcal W\wedge \sigma]
\end{equation}
where the notation for $[\sigma\wedge \mathcal Z]_{\mathcal Z}$ is defined in Proposition 2.5, and
 $\mathcal W$,  to abuse the notation,  denotes the current $\in\mathscr D'(\mathcal Z)$. 

\end{proposition}
\bigskip

\begin{proof}  Let $j: E\to \mathcal Z$ be a tubular neighborhood of $\mathcal Z$ in $\mathcal X$. 
Thus $E$ is diffeomorphic  to a vector bundle of rank $r$. We denoted the bundle also by $E$.  
Let $\mathcal U$ be a De Rham data for $\mathcal Z$ such that each De Rham chart $U_i$ lies in the trivialization. So
\begin{equation}
j^{-1}(U_i) =: U_i\times \mathbb R^r
\end{equation}
where $=:$ denotes a fixed diffeomorphism. Let $B\subset \mathbb R^r$ be the De Rham chart in the step 3 of the
construction of De Rham's regularization in [9].  Then $E$ is equipped with a De Rham data $\mathcal U_E$ whose De Rham covering is 
$j^{-1}(U_i), all\ i$.   Then we extend  $\mathcal U_E$ (arbitrarily) to 
$\mathcal X$ to have a De Rham data $\mathcal U_\mathcal X$.  We denote the smoothing 
operator associated to the chart $j^{-1}(U_i)$ by
$R_\epsilon^{\mathcal X, i}$. 
Let $\xi\in \mathscr D(\mathcal X)$ be a function that is 1 on $\mathcal W$ and 
has support in a neighborhood of $\mathcal Z$ inside of $E$.
Let $\mathcal W'=\xi j^{-1}(\mathcal W)$ be the current $\in \mathscr D'(\mathcal X)$.
By the associativity,
\begin{equation}
   [\mathcal W'\wedge [\mathcal Z\wedge \sigma]]=[[\mathcal W'\wedge \mathcal Z]\wedge \sigma].
\end{equation}
Notice that with vector bundle structure, $\mathcal W'$ in a neighborhood of $\mathcal W$ meets $\mathcal Z$ transversely at the 
open cells of $\mathcal W$.  By Proposition 3.1 below, $$[\mathcal W'\wedge \mathcal Z]=i_\ast\mathcal W.$$  so
\begin{equation}
  [\mathcal W'\wedge [Z\wedge \sigma]]=[i_\ast\mathcal W\wedge \sigma].\end{equation}

Let $\phi\in \mathscr D(\mathcal Z)$. By using the partition of unity, we may assume $\mathcal W$ lies in a single trivialization 
$$j^{-1}(U_i) =: U_i\times \mathbb R^r $$of
$E$.  Then we calculate  
\begin{align*}
& \int_{[ [Z\wedge \sigma]\wedge \mathcal W']}j^\ast(\phi) \\
&= \displaystyle{\lim_{(\epsilon, \epsilon')\to \mathbf 0}}
\int_{\mathcal Z} R_\epsilon^{\mathcal X}(\sigma)\wedge R_{\epsilon'}^{\mathcal X}(\mathcal W')\wedge j^\ast(\phi)\\
&= \displaystyle{\lim_{(\epsilon, \epsilon')\to \mathbf 0}}
\int_{\mathcal Z} R_\epsilon^{\mathcal X}(\sigma)\wedge R_{\epsilon'}^{\mathcal X, i_1}\circ\cdots \circ R_{\epsilon'}^{\mathcal X, i_h} (\mathcal W')\wedge j^\ast(\phi)\\ &
(\text {where $i_1<\cdots <i_h$, and apply Claim 2.9 to the trivialization (2.43) })
\end{align*}
\begin{align}\begin{split}
& =\displaystyle{\lim_{(\epsilon, \epsilon')\to \mathbf 0}}
\int_{\mathcal Z} R_\epsilon^{\mathcal X}(\sigma)\wedge (1\otimes R_{\epsilon'}^{\mathcal Z, i_1}\circ\cdots \circ R_{\epsilon'}^{\mathcal Z, i_h} (\mathcal W))\wedge j^\ast(\phi)\\
&=\displaystyle{\lim_{(\epsilon, \epsilon')\to \mathbf 0}}
\int_{\mathcal Z} R_\epsilon^{\mathcal X}(\sigma)\wedge (1\otimes R_{\epsilon'}^{\mathcal Z}(\mathcal W))\wedge j^\ast(\phi)\\
&=\displaystyle{\lim_{ \epsilon'\to  0}}
\int_{[\mathcal Z\wedge \sigma]} (1\otimes R_{\epsilon'}^{\mathcal Z}(\mathcal W))\wedge j^\ast(\phi)\\
&=
\int_{[[\mathcal Z\wedge \sigma]\wedge i_\ast\mathcal W]} j^\ast (\phi)\\
&=\int_{[[\mathcal Z\wedge \sigma]_\mathcal Z\wedge \mathcal W]} \phi
\end{split}\end{align}
Hence $$[ \mathcal W'\wedge  [Z\wedge \sigma]]=i_\ast [\mathcal  W\wedge[\mathcal Z\wedge \sigma]_{\mathcal Z}].$$
Combining with (2.45), we complete the proof. 
\end{proof}
\bigskip

\begin{definition} Any pair of De Rham data such as $\mathcal U_{\mathcal Z}, \mathcal U_{\mathcal X}$ satisfying  
(2.42) for any $\mathcal W$ will be called the inclusion De Rham data.

\end{definition}
\bigskip

\bigskip

\section{ Dependence of the local data}

\par

The intersection of currents is extrinsic because De Rham data is extrinsic. The real intersection theory
stresses inseparable nature of ``intrinsic" and ``extrinsic".  
  Let's look into the dependence of the
De Rham data to see how much is intrinsic and how much is extrinsic. We begin with the real case.
 \bigskip

\subsection  {  Real case}
\medskip

It is well-known that on a manifold, if two sub manifolds meet transversally at another submanifold, then the intersection
 should be defined as the intersectional manifold.  The more useful version is its extension in algebraic geometry.  The following proposition says that the tranversal intersection is a  particular case of the intersection of currents, where
 the dependence of De Rham data disappears. 
\bigskip

\begin{proposition} Let $\mathcal X$ be a manifold of dimension $m$.
If $T_1, T_2$ are cells of real dimension $p, q$ with $p+q\geq m$,  and the intersection $T_1\cap T_2$ is transversal at a connected, 
 manifold
$V$ of dimension $p+q-m$.  We assume $V$  at each point can be oriented concordant with $T_1, T_2$.  Then 
 $[T_1\wedge T_2]$ is independent of $\mathcal U$. Furthermore  it is the current of integration over $V$.

\end{proposition}

 \bigskip

\begin{proof} Let's set up the coordinates for the cells. 
Let  $\mathcal  X=\mathbb R^{m}$ have
 linear basis $\mathbf e_1, \cdots, \mathbf e_m$ 
and coordinates $x_1, \cdots, x_{m}$. 
Set up the subspaces, 
\begin{align*} & \mathbb R^p=span(\mathbf e_1, \cdots, \mathbf e_p)\\
 &\mathbb R^q=span(\mathbf e_{m-q+1}, \cdots,\mathbf  e_m)\\
& \mathbb R^{p+q-m}=span(\mathbf e_{m-q+1}, \cdots, \mathbf e_{p})
\end{align*}
Let $T_1=\Delta^p\subset \mathbb R^p$ be the polyhedron defined by
\begin{equation}
\{\sum_{i=1}^p |x_i| <1\}
\end{equation}
Similarly $T_2=\Delta^q$ is defined by 
\begin{equation}
\{\sum_{i=m-q+1}^m |x_i| <1\},
\end{equation}
 $V=\Delta^p\cap \Delta^q$ is defined by 
\begin{equation}
\{\sum_{i=m-q+1}^{p} |x_i |<1\}.
\end{equation}
Let $\pi_{p+q-m}: \mathbb R^m\to V$ be the projection. 
The proof has two steps.\medskip

$\mathbf 1 st$ step: 
Notice $[T_1\wedge T_2]$ has a compact support, hence 
$[T_1\wedge T_2]$ is also evaluated at the forms in $C^\infty(\mathbb R^m)$.
Let $\phi\in C^\infty(\mathbb R^m)$ have degree $p+q-m$.  Let $\xi\in \mathscr D_0(\mathbb R^m)$ that is equal to $1$ 
on a bounded  neighborhood $K$ of $\Delta^p\cap \Delta^q$.  We consider a $C^\infty$ form
\begin{equation}
\xi (\phi-\pi_{p+q-m}^\ast(\phi|_{V}))
\end{equation}
which has a compact support in $K$ and is a limit of $C^\infty$ forms
$\psi_n\in \mathscr D(K)$ compactly supported in $K-\bar V$ as $n\to \infty$. 
Because $[T_1\wedge T_2]$ is supported on $V$, but $\psi_n$ is supported outside of $V$,  
$$\int_{[T_1\wedge T_2]}\psi_n=0$$ for all $n$. 
By the continuity of the functional $[T_1\wedge T_2]$, 
\begin{equation}
\int_{[T_1\wedge T_2]}\phi=\int_{[T_1\wedge T_2]}\pi_{p+q-m}^\ast(\phi|_{V}).
\end{equation}
Recall that the $C^\infty$ form $\pi_{p+q-m}^\ast(\phi|_{V})$ is called a local constant slicing in [9]. Therefore it is a closed form.

\bigskip

Now we apply the homotopy formula (2.7). Let $\phi\in \mathscr D(\mathbb R^m)$ such that
$$supp(\phi)\cap \bigl(\partial (\Delta^p)\cup \partial (\Delta^q)\bigr)=\varnothing.$$
For arbitrary De Rham's regularization  $R'_\epsilon, A_\epsilon'$ with fixed sufficiently small real numbers $\epsilon_1, \epsilon_2$, we apply the homotopy formula (2.7) to have 
\begin{align}
& \int_{[T_1\wedge T_2]}\pi_{p+q-m}^\ast(\phi|_{V})\\ &
=\int_{[(bA_{\epsilon_1}'T_1+A_{\epsilon_1}' bT_1)\wedge (bA_{\epsilon_2}'T_2+A_{\epsilon_2}' bT_2)] }\pi_{p+q-m}^\ast(\phi|_{V})
\\ &+ \int_{[R_{\epsilon_1}'T_1\wedge R_{\epsilon_2}'T_2] }\pi_{p+q-m}^\ast(\phi|_{V})
\end{align}
Now we calculate the first integral  (3.7)
\begin{align*}
 & \int_{[(bA_{\epsilon_1}'T_1+A_{\epsilon_1}' bT_1)\wedge (bA_{\epsilon_2}'T_2+A_{\epsilon_2}' bT_2)] }\pi_{p+q-m}^\ast(\phi|_{V})
\\
&=\displaystyle{\lim_{\epsilon\to 0}}   
\int_{[R_\epsilon (bA_{\epsilon_1}'T_1+A_{\epsilon_1}' bT_1)\wedge R_\epsilon ( bA_{\epsilon_2}'T_2+A_{\epsilon_2}' bT_2)] }\pi_{p+q-m}^\ast(\phi|_{V})\\
& (\text {because  $supp(bT_i)\cap supp(\phi)=\varnothing$.})\\ 
&=\displaystyle{\lim_{\epsilon\to 0}}   
\int_{[b R_\epsilon A_{\epsilon_1}'T_1\wedge  b R_\epsilon A_{\epsilon_2}'T_2] }\pi_{p+q-m}^\ast(\phi|_{V})\\
&=\pm \displaystyle{\lim_{\epsilon\to 0}}   
\int_{[R_\epsilon A_{\epsilon_1'} T_1\wedge   R_\epsilon A_{\epsilon_2}'T_2] } d (\pi_{p+q-m}^\ast(\phi|_{V}))
\\& =0
\end{align*}

This shows that
\begin{equation} 
\int_{[T_1\wedge T_2]}\pi_{p+q-m}^\ast(\phi|_{V})
=\int_{[R_{\epsilon_1}'T_1\wedge R_{\epsilon_2}'T_2] }\pi_{p+q-m}^\ast(\phi|_{V}).
\end{equation}
We observe that the right hand side of (3.9) does not involve the De Rham's smoothing operator $R_\epsilon$, thus 
 the current $[T_1\wedge T_2]$ is independent of the choice of De Rham data $\mathcal U$.
 \footnote {The technique of using the arbitrary $R_\epsilon'$  is due to  De Rham \S 20,  [2].}
\medskip

$\mathbf 2 nd$ step:  To calculate the intersection $[T_1\wedge T_2]$.  By the 1st step,  we can choose a particular  De Rham's data $\mathcal U$ that has 
one chart $x_1, \cdots, x_m$ for $\mathbb R^m$. Also we choose a $C^\infty$ convolution function $f(\mathbf x)$ supported in a neighborhood of
a unit ball $B$ satisfying
\begin{equation}
\int_{\mathbb R^m}f(\mathbf x)d\mu=1.
\end{equation}
where $d\mu=dx_1\wedge\cdots\wedge dx_m$. 
Let $\vartheta_\epsilon(x)=f ({\mathbf x\over\epsilon})  d{\mu\over \epsilon}$. 

Let
\begin{equation}\begin{array}{ccc}
\kappa: \mathbb R^m\times \mathbb R^m &\rightarrow & \mathbb R^m\\
(\mathbf x, \mathbf y) &\rightarrow &\mathbf  x-\mathbf y ,
\end{array}\end{equation}
Denote the coordinates  $(x_1, \cdots, x_{m-q})$ by $ \mathbf x_1$, $(x_{m-q+1}, \cdots, x_p)$ by $\mathbf x_2$ and
$x_{i+1}, \cdots, x_m$ by $\mathbf x_3$. 
Similarly for the second copy of $\mathbb R^m$ in (3.11), the corresponding coordinates  
are denoted by $\mathbf y_1, \mathbf y_2, \mathbf y_3$ respectively. 

Let \begin{equation}g({\mathbf x_1\over \epsilon}, {\mathbf x_2\over \epsilon}, {\mathbf x_3\over \epsilon}, 
{\mathbf y_1\over \epsilon}, {\mathbf y_2\over \epsilon}, {\mathbf y_3\over \epsilon})=\kappa^\ast (\vartheta_\epsilon)
.\end{equation}
Let $\phi\in \mathscr D(\mathbb R^m)$ be a test form. 
 Then we calculate 
the current 

\begin{align}\begin{split} 
 &\int_{[T_1\wedge T_2]}\phi=\lim_{\epsilon\to 0}\int_{T_1} R_\epsilon(T_2)\wedge \phi\\
&=\lim_{\epsilon\to 0}\int_{T_1} \int_{(\mathbf y_2, \mathbf y_3)\in T_2}
g({\mathbf x_1\over \epsilon}, {\mathbf x_2\over \epsilon}, 0, 
0, {\mathbf y_2\over \epsilon}, {\mathbf y_3\over \epsilon}) \wedge \phi(\epsilon {\mathbf x_1\over \epsilon} , \mathbf x_2, 0 )
\end{split}\end{align} 
where $\phi(\epsilon {\mathbf x_1\over \epsilon} , \mathbf x_2, 0 ) $
is a test form, i.e. $C^\infty$ form on $T_1$ with a compact support.  
Now applying   the fibre integral to that over $T_1$,  we obtain 
\begin{align}\begin{split}
 & \int_{[T_1\wedge T_2]}\phi\\
&=\lim_{\epsilon\to 0}\int_{\mathbf x_2\in \mathbb R^{i+j-m}}\int_{ \mathbf x_1\in \mathbb R^{m-j}}
\int_{(\mathbf y_2, \mathbf y_3)\in \mathbb R^j}
g({\mathbf x_1\over \epsilon}, {\mathbf x_2\over \epsilon}, 0, 
0, {\mathbf y_2\over \epsilon}, {\mathbf y_3\over \epsilon})\wedge \phi(\epsilon {\mathbf x_1\over \epsilon} , \mathbf x_2, 0 )
\end{split}\end{align} 
Then we make a change of variables,
\begin{equation}\begin{array}{c}
{\mathbf x_1\over \epsilon}\to \mathbf x_1, \\
{\mathbf y_2\over \epsilon}\to \mathbf y_2\\
{\mathbf y_3\over \epsilon}\to \mathbf y_3.
\end{array}\end{equation}
Then \begin{align}\begin{split}
 & \int_{[T_1\wedge T_2]} \phi\\
 & =\pm  \lim_{\epsilon\to 0}\int_{\mathbf x_2\in \mathbb R^{i+j-m}}\int_{ \mathbf x_1\in \mathbb R^{m-j}}
\int_{(\mathbf y_2, \mathbf y_3)\in \mathbb R^j}
g(\mathbf x_1, {\mathbf x_2\over \epsilon}, 0, 
0, \mathbf y_2, \mathbf y_3)\wedge \phi(\epsilon \mathbf x_1, \mathbf x_2, 0 )
\end{split}\end{align} 
 Then we notice for each fixed $\mathbf x_2$, 
 the fibre integral
\begin{equation}  \int_{\mathbf y_2, \mathbf y_3\in \mathbb R^j, \mathbf x_1\in \mathbb R^{m-j}}
g({\mathbf x_1}, {\mathbf x_2\over \epsilon}, 0, 
0, {\mathbf y_2}, {\mathbf y_3})
\end{equation}
by formula (3.10),   is 1.  Therefore we obtain that 
\begin{equation}
[T_1\wedge T_2](\phi)=\int_{ \mathbb R^{i+j-m}}\phi(0, \mathbf x_2, 0)
\end{equation}

Thus
\begin{equation}\begin{array}{c} 
[T_1\wedge T_2](\phi)=\int_{V} \phi|_V.
\end{array}\end{equation} 
where $\phi|_V$ is the restriction $\phi(0, \mathbf x_2, 0)$ of $\phi$ to $V$. 
We complete the proof.

\end{proof}

For non transversal intersection, there is no classification in general situation because the intersection may not have 
the notion of dimension. But for some specific  examples, the notion of the dimension exists. 

\bigskip

\begin{proposition} (real excess intersection)
The  intersection of currents $[T_1\wedge T_2]$ depends on $\mathcal U$.

\end{proposition}
\bigskip

\begin{proof}  We prove it by using an example. 
Let $\mathcal X=\mathbb R^2$, and  be equipped with the De Rham data consisting of  single chart $\mathbb R^2$ 
with  the convolution  function 
$f$ satisfying 
\begin{equation}
\int_{\mathbb R^2} f (x_1, x_2)dx_1\wedge dx_2=1
\end{equation}
where $x_1, x_2$ are Euclidean coordinates of $\mathbb R^2$.  Let $T_1=T_2$ be the current of integration 
over the finite piece of the parabola 
\begin{equation}
x_1=x_2^2
\end{equation}
containing the origin $\bf 0$.
Since $T_1, T_2$ are singular chains,  
 $[T_1\wedge T_2]$ exists.  
Let $\phi(x)$ be a test function with a compact support.  Denote the second copy of $\mathbb R^2$ for the De Rham's regularization 
by $y_1, y_2$.
Then we calculate

\begin{equation}\begin{array}{c}
\int_{[T_1\wedge T_2]}\phi\\
\|\\
\displaystyle{\lim_{\epsilon\to 0}} {1\over \epsilon^2}\int_{x\in T_1}\int_{y\in T_2}
f({x_1-y_1\over \epsilon}, {x_2-y_2\over \epsilon})\phi(x_1, x_2) (dx_1-dy_1)\wedge (dx_2-dy_2) 
\end{array}\end{equation}
substitute $x_1=x_2^2, y_1=y_2^2$ for $T_1, T_2$, we obtain that

\begin{equation}\begin{array}{c}
\int_{[T_1\wedge T_2]}\phi\\
\|
\\
\displaystyle{\lim_{\epsilon\to 0}} {2\over \epsilon^2}\int_{x_2\in \mathbb R}\int_{y_2\in \mathbb R}
f({(x_2-y_2)(x_2+y_2)\over \epsilon}, {x_2-y_2\over \epsilon})\phi(x_1, x_2) (x_2-y_2)dy_2\wedge dx_2.\end{array}
\end{equation}
Next we make a change of the variables
\begin{equation}\begin{cases}
u={(x_2-y_2)\over \epsilon}\\
v=x_2+y_2
.\end{cases}
\end{equation}
Then 
\begin{equation}\begin{array}{c}
[T_1\wedge T_2](\phi)\\
\|
\\
\displaystyle{\lim_{\epsilon\to 0}} \int_{u\in \mathbb R}\int_{v\in \mathbb R}
u f(uv, u)  \phi ( ({\epsilon u+v\over 2})^2, {\epsilon u+v\over 2}) dv\wedge du\\
\|\\
 \int_{(u, v)\in \mathbb R^2}
u f(uv, u) \phi ( ({v\over 2})^2, {v\over 2}) dv\wedge du
.\end{array}
\end{equation}
Then the functional \begin{equation}
\phi\to \int_{(u, v)\in \mathbb R^2}
u f(uv, u) \phi ( ({v\over 2})^2, {v\over 2}) dv\wedge du\end{equation}
defines a current supported on $T_1$.   So the intersection current 
$$[T_1\wedge T_2]$$(which is (3.26))
  is
  supported on $T_1$, depending  on the convolution  function $f$. 

\end{proof}

\bigskip

\bigskip

\begin{ex} (real proper intersection)

Let $\mathcal X=\mathbb R^2$  be equipped with the De Rham data consisting of  single chart $\mathbb R^2$ with  the convolution  function 
$f$ satisfying 
\begin{equation}
\int_{\mathbb R^2} f (x_1, x_2)dx_1\wedge dx_2=1
\end{equation}
where $x_1, x_2$ are Euclidean coordinates of $\mathbb R^2$.
\bigskip

Case 1: Let $T_1$ be  a line through the origin $\mathbf 0$ and 
$T_2$ is another line segment through the origin.  Then it is known that
$$[T_1\wedge T_2]=\delta_{\mathbf 0}$$
if the order matches with the orientation of $\mathbb R^2$.
(for instance see Proposition 3.1).

\bigskip

Case 2: Continuing from the setting in case 1, 
let $T_2$ be the line $x_1=0$. Let $T_1$ be a pieces of  parabola
\begin{equation}
x_1=x_2^2, x_2\in (-1, 1).
\end{equation}
Let's calculate $[T_1\wedge T_2]$. 
Let $\phi(x)$ be a test function supported in a neighborhood of  the origin.  
We denote the second copy of $\mathbb R^2$ for De Rhams' regularization by $(y_1, y_2)$. Then 
\begin{align}\begin{split} 
 &\int_{[T_1\wedge T_2]}\phi\\
&=
\displaystyle{\lim_{\epsilon\to 0}} {1\over \epsilon^2}\int_{x_1\in T_1}\int_{y_2\in \mathbb R}
f({x_1\over \epsilon}, {x_2\over \epsilon}-{y_2\over \epsilon})\phi(x_1, x_2) dy_2\wedge dx_1.
\end{split}\end{align}
Let  $$f_1(x_1)=\int_{y_2\in\mathbb R} f(x_1, - y_2) dy_2.$$
Now we continue (3. 29) to have 
\begin{equation}\begin{array}{c}
\int_{[T_1\wedge T_2]}\phi\\ \|\\
\displaystyle{\lim_{\epsilon\to 0}} {1\over \epsilon}\int_{(x_1, x_2)\in T_1}
f_1({x_1\over \epsilon})\phi(x_1, x_2) dx_1\\
\|\\
\phi(\mathbf 0)\biggl (\int_{+\infty}^{0} f_1(x_1)dx+\int_{0}^{+\infty} f_1(x_1)dx_1\biggr)=0,
\end{array}\end{equation}

So 
$$[T_1\wedge T_2]=0$$ for all convolution  function $f$ in the De Rham data.  This example shows
the formula
$$supp([T_1\wedge T_2])=supp(T_1)\cap supp(T_2)$$
does not hold for singular chains. \bigskip

Case 3: Continuing from the setting in case 2, 
 let $T_2$ be the line $x_1=0$. Let $T_1$ be a piece of  the cubic curve
\begin{equation}
x_1=x_2^3, x_2\in (-1, 1).
\end{equation}
The same calculation in case 2  shows if order of $T_1, T_2$ is concordant with orientation of $\mathbb R^2$, then
\begin{equation}
\int_{[T_1\wedge T_2]}\phi=\phi(\mathbf 0).
\end{equation}
Hence 

\begin{equation}
[T_1\wedge T_2]=\delta_{\mathbf 0}
\end{equation}
where $\delta_{\mathbf 0}$ is the $\delta$-function at the origin. So the intersection is independent choice of
convolution  function $f$ in De Rham data. 
\end{ex}

\bigskip

{\bf Remark} All three cases in Example 3.3 belong to the case of ``real proper intersection" (not fully defined) which by the step 1 of 
Proposition 3.1  is independent of De Rham data.  They also  coincide with Kronecker index $T_1\wedge T_2[1]$ defined by
De Rham. 
However the multiplicity associated to the ``proper" components 
 is different from that of the complex cases.   Thus such an intrinsic  problem of determination 
of the multiplicity has no satisfactory answer.

\bigskip

\subsection{  Complex case}

\bigskip

\bigskip

\begin{proposition} Let $f: X\to Y$ be a regular map between two smooth projective varieties.
Let $W$ be a $p$ dimensional algebraic cycle of $ X$.
Then the current $f_\ast [W]$ is the current of integration over
the cycle 
$$f_\ast W,$$
where $[W]$ stands for the current of integration over the algebraic set.

\end{proposition}

\bigskip

\begin{proof}
Let $W=\sum_i a_iW_i$ where $W_i$ are irreducible and $a_i$ are non-zero integers.  Let
  $f_\ast W=\sum_i b_i S_i$ 
where $b_i$ are non-zero integers divisible by non-zero $a_i$ and $S_i$ are irreducible subvarities.
Let $|W_0|$ be the open sets of the support  $|W|$ such that $f$ is smooth.
Then correspondingly $f(|W_0|)=\cup_i S_i^0$, where $S_i^0$ are open sets of $S_i$. 
By the definition of the push-forward of algebraic cycles, 
the map
\begin{equation}\begin{array}{ccc}
f: W_i &\rightarrow & S_i
\end{array}\end{equation}
is a finite to one morphism, and $f$ is restricted to an e\'tal morphism on $f^{-1}(S_i^0)$. 
Then using currents, we have 
\begin{equation}
f_\ast [f^{-1}(S_i^0)]={b_i\over a_i} S_i^0.
\end{equation}
Taking the closure and the sum  over $i$, we obtain
that
\begin{equation}
 f_\ast (\sum_i a_i[W_i])=\sum_i b_i [S_i].
\end{equation}
Since for algebraic cycles, we have
\begin{equation}
 f_\ast (\sum_i a_i W_i )=\sum_i b_i S_i,
\end{equation}
we complete the proof. 

\end{proof}

\bigskip

\begin{theorem} Let $X$ be a smooth projective variety of dimension $n$ over $\mathbb C$. 
 Let $T_1, T_2$ be irreducible subvarieties of $X$ of dimension $p, q$.  To abuse the notations, 
the currents of integration over them are also denoted by $T_1, T_2$
respectively.  Assume  $T_1\cap  T_2$ is proper. Then with an arbitrary  De Rham data $\mathcal U$ on $X$, 
the current $[T_1\wedge T_2]$ is independent of $\mathcal U$, and  equals to the current of integration
over the algebraic cycle $$T_1\cdot T_2,$$
where $T_1 \cdot T_2$ is the Fulton's intersection ([3]) defined as the linear combination of
all irreducible components of the scheme 
$$T_1\cap T_2.$$
Furthermore the assertion  extends to algebraic cycles linearly. 
\end{theorem}
\bigskip

{\bf Remark} Theorem shows that $[T_1\wedge T_2]$ in this case is intrinsic. This is the principle for the real intersection theory:
every intrinsically defined intersection in the past can be interpreted as an intersection of currents 
that are independent of De Rham data.

\bigskip

\begin{proof}   
Let's fix the cycle $T_2$.  By example 11.4.2, [3], there is an algebraic cycle $E_1$ rationally equivalent to $T_1$ such that
$A$ meets $T_2$ transversely (at an open set of each irreducible support).  Without losing the generality,
 let's have a  simplified setting as follows.
Let $V\subset X\times \mathbf P^1$ be an irreducible subvariety, and $P_1: V\to X$, $P_2: V\to \mathbf P^1$ are the projections.
Let $T_2\subset X$ be an irreducible subvariety.   Assume the cycle of the scheme $P_2^{-1}(1)$ is $E_1$ and the cycle of the scheme
$P_2^{-1}(0)$ is $T_1$, where $0, 1$ are two points of $\mathbf P^1$. Let $I$ be a real curve in $\mathbf P^1$  connecting 
$0, 1$. Next we consider two objects: currents and algebraic cycles. 
Using the currents, according to Proposition 4.12 below in section 4, we have 
the formula 
\begin{equation}
[T_1\wedge T_2]-[E_1\wedge T_2]=d \biggl ( (P_1)_\ast ( V_I(T_2))\biggr)
\end{equation} 
where $ V_I(T_2)$ is the current
\begin{equation}
[V\wedge (T_2\otimes I)].
\end{equation}

Next we consider the algebraic cycles in  intersection theory where two rationally equivalent algebraic cycles
are homotopic. More precisely we have the equation in singular cycles
\begin{equation}
T_1\cdot T_2-E_1\cdot T_2=d \biggl ( (P_1)_\ast ( V_I'(T_2))\biggr),
\end{equation}
where $V_I'(T_2)$ is the singular cycle defined by the semi-algebraic set in the complex manifold $X\times\mathbf P^1$, 
$T_1\cdot T_2, E_1\cdot T_2$ are singular cycles obtained from the triangulation of the intersectional algebraic cycles,
and $d$ is the differential operator on the singular chains ( the boundary operator with a sign).   
We should note that  the current of integration over 
$V_I'(T_2)$ is the current $V_I(T_2)$. 
Next we convert the equation (3.40) to that in currents to have
\begin{equation}
[T_1\cdot T_2]-[E_1\cdot T_2]=d \biggl ( (P_1)_\ast ( V_I(T_2))\biggr). 
\end{equation}
Hence 
\begin{equation}
[T_1\wedge T_2]-[E_1\wedge T_2]=[T_1\cdot T_2]-[E_1\cdot T_2]
\end{equation}
(where the difference of two sides should be noticed).
By Proposition 3.1, since $E_1$ meets $T_2$ transversely,  \begin{equation}
[E_1\wedge T_2]=[E_1\cdot T_2].\end{equation}
Thus $$
[T_1\wedge T_2]=[T_1\cdot T_2].
$$
We complete the proof. 

\end{proof}

\bigskip

\bigskip

\begin{theorem}
Let $X$ be a smooth projective variety of dimension $n$ over $\mathbb C$. 
 Let $T_1, T_2$ be  subvarieties of $X$ of codimension $p, q$.  
The currents of integration over them are also denoted by $T_1, T_2$
respectively.  Assume $T_1\cap T_2$ is an excess intersection.  Then
\begin{equation}
[T_1\wedge T_2]
\end{equation}
in general depends on the De Rham data $\mathcal U$.

\end{theorem}

\bigskip

\begin{proof}  
Let's give an example where $[T_1\wedge T_2]$ is dependent of De Rham data.
Let $\mathbf P^2$ be a projective space over $\mathbb C$  with affine coordinates $(z_1, z_2)$. 
  Let $T_1$ be the  hyperplane $z_2=0$, and $T_2=T_1$.  First it is not zero because its reduction to cohomology group is non-zero.
Choose two open sets as De Rham's covering:  $U_1$, the finite affine plane, and a small neighborhood $U_2$ of the infinity $\mathbf P^1\subset \mathbf P^2$. 
Choose real Euclidean coordinates $x_1, y_1, x_2, y_2$ for $U_1$ such that $$z_1=x_1+iy_2, z_2=x_2+iy_2.$$
Use these open covering and Euclidean coordinates to have a De Rham data for $\mathbf P^2$ with a 
convolution  function $h(x_1, x_2, y_1, y_2)$ of the unit ball $B$ in $U_1$.
Then we see in $U_1$, 
\begin{equation}
R_\epsilon ^1(T_2)=-{1\over \epsilon^4}\iint_{(x_1', y_1')\in \mathbb R^2}
h({x_1-x_1'\over \epsilon}, {x_2\over \epsilon}, {y_1-y_1'\over \epsilon}, {y_2\over \epsilon})
dx_1'\wedge dy_1'\wedge dx_2\wedge dy_2,\end{equation}
where $x_i', y_i'$ are the Euclidean coordinates for the second factor in the smoothing operator. 
The composing with another local smoothing operator from $U_2$ will not change the 
smooth current  $R_\epsilon ^1(T_2)$ in $B$. 
Thus for a test form $\phi$ supported in $B$, the integral 
\begin{equation}
\int_{T_1}R_\epsilon^B (T_2)\wedge \phi=\iint_{x_2=y_2=0} (\cdots) dx_2\wedge dy_2=0.
\end{equation}
This shows with this type of De Rham data, 
\begin{equation}
[T_1\wedge T_2]
\end{equation}
is zero on $U\cap T_1$.  Hence $[T_1\cap T_2]$ is a $0$-dimensional current  supported 
at the infinity point of $T_1$.  Since the $\infty=\mathbf P^1$ is  arbitrary,   
 $[T_1\wedge T_2]$ is supported on an arbitrary  set determined by the De Rham data.

\end{proof}

\begin{ex}\quad \par
The  intersection of currents of integration over algebraic cycles  always exists because algebraic cycles are Lebesgue. But its
intersection depends on De Rham data. Theorem 3.5 asserts that in case of a proper intersection, 
it is actually independent of
De Rham data.  
But for excess intersection, the situation still goes back to the dependence of De Rham data, 
and the interpretation is   different  from Fulton's.  
\bigskip

\begin{table}[ht]
\caption{\large {Excess intersection in complex case}} \medskip
\begin{tabular}{|c|| c |c|c|}

\hline
 Intersection &  Cycle &  Cycle class   & Support 
\\
\hline
&&&\\
algebraic\ $T_1\cdot T_2$ & not well-defined & well-defined in the Chow ring, & $T_1\cap T_2$\\
 &  & and cohomology ring & \\  
\hline  &&&\\

current \ $[T_1\wedge T_2]$ &  well-defined,   & not well-defined in the Chow ring,  & $T_1\cap T_2$ \\
 & but $\mathcal U$-dependent  & well-defined in cohomology ring &\\
 \hline
\end{tabular}
\label{table:nonlin} 
\end{table}

 \end{ex}

\bigskip

$$$$

\bigskip

\section{Tools for application}

\subsection{ Correspondence of a current }\par

\bigskip

\begin{lemma} Let $\mathcal X, \mathcal Y$ be two compact  manifolds, and 
$P_{\mathcal X}$ be the projection 
$$\mathcal X\times \mathcal Y\to \mathcal X.$$
Then the image of the projection
$$\begin{array}{ccc}
(P_{\mathcal X})_\ast: \mathcal C(\mathcal X\times \mathcal Y) &\rightarrow & \mathscr D'(\mathcal X)
\end{array}$$
lies in $\mathcal C(\mathcal X)$.

\end{lemma}

\bigskip

\begin{proof}
Notice there is a coordinates  chart  of $\mathcal X\times \mathcal Y$ satisfying that  the coordinates planes of $\mathcal X$ are also 
the coordinates  planes for $\mathcal X\times \mathcal Y$. Thus the two conditions of Lebesgue currents for $\mathcal X$
are implied by that for $\mathcal X\times \mathcal Y$.

\end{proof}

\bigskip

\begin{definition} Let $\mathcal X, \mathcal Y$ be two compact  manifolds. 

Let \begin{equation}
F \in \mathcal C( \mathcal X \times  \mathcal Y)
\end{equation}
be a Lebesgue current. 
Let $\mathcal U$ be a product De Rham data on $\mathcal X \times  \mathcal Y$.  
Let $P_\mathcal X, P_\mathcal Y$ be the projections 
$$\mathcal X\times \mathcal Y\to \mathcal X, \quad  \mathcal X\times \mathcal Y\to \mathcal Y.$$

Define pull-back of currents
$$ F^\ast(T)$$ by 
\begin{equation}\begin{array}{ccc}
F^\ast: \mathcal C(\mathcal Y) &\rightarrow & \mathcal C(\mathcal X)\\
T &\rightarrow & (P_\mathcal X)_\ast [F\wedge ([\mathcal X\otimes T)].
\end{array}\end{equation}

Define the push-forward $F_\ast (T)$ of currents by

\begin{equation}\begin{array}{ccc}
F_\ast:  \mathcal C(\mathcal X) &\rightarrow & \mathcal C(\mathcal Y)\\
T  &\rightarrow &  (P_\mathcal Y)_\ast [ F\wedge (T\otimes [\mathcal Y])]
.\end{array}\end{equation}

\end{definition}

\bigskip

\bigskip

\begin{proposition} 
Let $X, Y, Z$ be three smooth projective varieties over $\mathbb C$. Let 
$F_{XY}, F_{YZ}$ be algebraic cycles on $X\times Y$ and $Y\times Z$ respectively, and  represent  finite  correspondences ([6]). Then
\begin{equation}
(F_{YZ})_\ast\circ (F_{XY})_\ast= (F_{YZ}\circ F_{XY})_\ast,
\end{equation}
where $F_{YZ}\circ F_{XY}$ denotes the multiplication of algebraic correspondences.
Similarly  \begin{equation}
(F_{XY})^\ast\circ (F_{YZ})^\ast= (F_{XY}^t\circ F_{YZ}^t)^\ast,
\end{equation}
where $F_{XY}^t,  F_{YZ}^t$ are the transposes of the correspondences. 
\end{proposition} 
\bigskip

\begin{proof}
Because they are finite correspondences, hence
\begin{equation}
(F_{XY}\times Z)\cdot (X\times F_{YZ})
\end{equation}
is a well-defined algebraic cycle $\Gamma$ in $X\times Y\times Z$.
Let $\sigma\in \mathcal C(X)$.
By Theorem 3.5, evaluations of both sides of (4.4) is equal to
\begin{equation}
(P_Z)_\ast [ (\sigma\otimes Y\otimes Z)\wedge \Gamma],
\end{equation}
where $P_Z: X\times Y\times Z\to Z$ is the projection. So (4.4) is proved. 
The  (4.5) is the transpose of (4.4).

\end{proof}
\bigskip

\begin{proposition} Let $ X,  Y$ be compact complex manifolds. 
The pull-back and push-forward of currents extend  
Gillet and Soul\'e's push-forward of currents and smooth pull-back  of currents.
\end{proposition}

\begin{proof}
In [4], Gillet and Soul\'e defined operations on the currents on compact complex manifolds. They include
push-forward for proper maps and pull-back for smooth maps. We verify that 
these operations coincide with ours. \bigskip
 
Let \begin{equation}\begin{array}{ccc}
f:  X &\rightarrow &  Y
\end{array}\end{equation}
be a regular map. Let $F$ be its graph.
Let $T$ be a Lebesgue current on $ X$.
Let  $\phi$ be a $C^\infty$ form on  $ Y$. We use product De Rham data on $ X\times  Y$ 
and projection  De Rham 
data on $ X $ and $  Y$. 
 Then 
\begin{align}\begin{split}
 &\int_{F_\ast(T)}\phi \\
& =
\lim_{\epsilon\to 0}\int_{F} R_\epsilon^{ X\times  Y}(T\times  Y)\wedge (P_ Y)^\ast (\phi)\\
 & (\text {by  Proposition 2.8, the  projection  formula })\\
 &=\lim_{\epsilon\to 0}\int_{F} ( P_{ X})^\ast R_\epsilon^{ X} (T)\wedge  (P_ Y)^\ast (\phi)\\
& =
\lim_{\epsilon\to 0}\int_{ X} R_\epsilon^{ X}(T)\wedge f^\ast(\phi)\\
 &=\int_{T} f^\ast(\phi) .
\end{split}\end{align}
This shows $$F_\ast (T)= f_\ast (T)$$
where $f_\ast$ is defined as the dual of the pullback on forms in 1.1.4, ([4]). 

\bigskip
Now 
let \begin{equation}\begin{array}{ccc}
f:  X &\rightarrow &  Y
\end{array}\end{equation}
be a smooth map.  Let  
$$F\subset  X \times   Y$$ be its graph.
Let $\phi$ be a test form on $ X$. 
\begin{align}\begin{split} &
\int_{F^\ast(T)}\phi=\int_{(P_ X)_\ast [F\wedge ( X\times T)]} \phi
\\&=
\displaystyle{\lim_{\epsilon\to 0}} \int_{F} R_\epsilon^{ X\times  Y}( X\times T)
\wedge (P_X)^\ast( \phi)\\
&   (\text {by  Proposition 2.8, the  projection  formula })\\
&=\displaystyle{\lim_{\epsilon\to 0}} \int_{F} (P_ Y)^\ast (R_\epsilon^{Y}(T))
\wedge (P_ X)^\ast( \phi)
.\end{split}\end{align}
Notice  
\begin{equation}\begin{array}{ccc}
P_{ Y}: F &\rightarrow &  Y
\end{array}\end{equation}
is isomorphic to $f$ which is also smooth.
Then we apply the fibre integral to have

\begin{equation}
\displaystyle{\lim_{\epsilon\to 0}} \int_{F} (P_ Y)^\ast (R_\epsilon^{ Y}(T))
\wedge (P_ X)^\ast( \phi)=\int_{T} f_\ast(\phi).
\end{equation}

Thus \begin{equation}
F^\ast(T)=f^\ast(T).
\end{equation} We
complete the proof. 
\end{proof}

\bigskip

\begin{proposition}\par

Let $\mathcal  X,  \mathcal Y$ be two compact manifolds. 

Let \begin{equation}
F \in \mathcal C(\mathcal   X \times  \mathcal  Y)
\end{equation}
be a homogeneous closed, Lebesgue current. 

(a)  
Let $T$ be a Lebesgue current of $\mathcal  X$ or $\mathcal Y$.  
 Then $supp (F_\ast (T))$ is contained in the set 
$$ P_{\mathcal Y} \biggl(supp (F)\cap (supp(T)\times \mathcal Y)\biggr);$$ $supp (F^\ast(T))$ is contained in
the set  $$ P_\mathcal X  \biggl(supp(F)\cap (\mathcal X\times supp(T) \biggr).$$
 \par

(b) If $T_1, T_2$ are Lebesgue and closed (rspt.  homologous to zero) in $\mathcal X$ and $\mathcal Y$ respectively, then
$F_\ast (T_1), F_\ast (T_2)$ are also closed (rspt. homologous to zero).

\end{proposition}

\bigskip

\begin{proof}
(a) Let $S$ be a Lebesgue current on $\mathcal X\times \mathcal Y$.
Let $a\notin P_\mathcal Y(supp(S))$. Then there is a neighborhood $B_a\subset \mathcal Y$ of a, such that
$$(\mathcal X\times B_a)\cap supp(S)=\varnothing.$$
Then for any $\phi\in \mathscr D(\mathcal Y)$ supported in $B_a$, 
\begin{equation}\int_S (P_\mathcal Y)^\ast(\phi)=0.\end{equation}
Then (4.16) says $$a\notin supp( (P_\mathcal Y)_\ast (S)).$$
Hence 
So \begin{equation}
supp( (P_\mathcal Y)_\ast (S))\subset P_\mathcal Y (supp (S)).
\end{equation}
Similarly 
 \begin{equation}
supp( (P_\mathcal X)_\ast (S))\subset P_\mathcal X (supp (S)).
\end{equation}
\bigskip

Now we consider our case.  Applying the assertion (4.17) for $$S=F\wedge (T\otimes \mathcal Y),$$
 together with part (1), property 2.6,

\begin{equation}\begin{array}{c}
supp ((P_\mathcal Y)_\ast[ F\wedge (T\times \mathcal Y)])\\
\cap\\
supp\biggl(
P_\mathcal Y(supp([ F\wedge (T\times \mathcal Y)]) )\biggr)\\
\cap\\
supp \biggl(P_\mathcal Y(supp( F)\cap (supp(T)\times \mathcal Y))\biggr).
\end{array}\end{equation}

The proof of 
\begin{equation}
supp(F^\ast(T))\subset  P_\mathcal X  \biggl(supp(F)\cap (\mathcal X\times supp(T) \biggr).
\end{equation} is similar.
\par

(b) By property 2.6, the currents $$[F\wedge (T_1\times \mathcal Y)], \ [F\wedge (\mathcal X\times T_2)]$$ are closed.
Therefore $F^\ast T_2, F_\ast T_1$ are closed.  If they are homologous to zero, 
then by the property 2.6, $$[F\wedge (T_1\times \mathcal Y)], \ [F\wedge (\mathcal X\times T_2)]$$
are homologous to zero in $\mathcal X, \mathcal Y$. Thus $F^\ast T_2, F_\ast T_1$ are homologous to zero.

We complete the proof

\end{proof}

\bigskip

\bigskip

\begin{ex}  Let $X, Y$ be two smooth projective varieties over $\mathbb C$, 
$$f: X\dasharrow Y$$ be a rational map. Then there is graph
 \begin{equation}
F\subset X\times Y.
\end{equation}

 Once $X\times Y$ is equipped with De Rham data (which does not have any requirements for $X, Y$), 
there are homomorphisms $F_\ast, F^\ast$ 
\begin{equation}\begin{array}{ccc}
F_\ast: \mathcal C(X) &\rightarrow & \mathcal C(Y)\\

 F^\ast: \mathcal C(Y) &\rightarrow & \mathcal C(X)
.\end{array}\end{equation}
When $\mathcal C(X), \mathcal C(Y)$ are reduced to cohomology, 
$F_\ast, F^\ast$ are reduced to the usual cohomological correspondences.

\end{ex}

\bigskip

\subsection{ Functoriality }

\bigskip

The real intersection theory is not functorial in the usual category of algebraic varieties.   However if we attach 
De Rham data, then it has a  functoriality. 
\bigskip

\begin{prodef}
Let $Cord(\mathbb C)$ be the category whose objects 
are the pairs of a smooth projective variety over $\mathbb C$,  and a De Rham data on it, denoted by $(X, \mathcal U)$. The 
morphisms are finite  correspondences of $X\times Y$.\end{prodef}
\bigskip

\begin{proof} 
The verification of the category is done  in [6].

\end{proof}

\bigskip

\begin{definition} Let $k$ be a whole number. 
Let $(X, \mathcal U)\in Cord(\mathbb C)$. 
Define  $\mathcal N_k\mathcal C(X)$ to be the linear span of Lebesgue currents  $$T\in \mathcal C(X)$$
satisfying
\par
(I)  $dim(T)\leq k$, OR $dim(T)\geq 2n-k$ OR,\par
(II) $supp(T)$ lies in an algebraic set $A$ of dimension
$$\leq {dim(T)+k\over 2}.$$

\bigskip

A current in  $\mathcal N_k\mathcal C(X)$ will  be called $\mathcal N_k$ leveled. The objects $\mathcal N_k\mathcal C(X)$ with usual
group homomorphisms form a subcategory of the Abelian category, denoted by $\mathcal N_k\mathcal C$.

\end{definition}

\bigskip

{\bf Remark} De Rham data  $\mathcal U$ plays no role in both categories $Cord(\mathbb C)$ and $\mathcal N_k\mathcal C$.  
However the functoriality  needs De Rham data. 
\bigskip

\begin{prodef}
The maps
\begin{equation}\begin{array}{ccc}
(X, \mathcal U) &\rightarrow &\mathcal N_k\mathcal C(X)

\end{array}\end{equation}
and 
\begin{equation}\begin{array}{ccc}
Hom\ in \ Cord(\mathbb C)  &\rightarrow & Hom \ in\ \mathcal N_k\mathcal C \\
\Gamma &\rightarrow & \Gamma_\ast, \quad  (covaraint)\\
\Gamma &\rightarrow & \Gamma^{\ast}, \quad (contrvariant)
\end{array}\end{equation}
\noindent define a covariant functor and a contravariant functor.
\end{prodef}

\bigskip

\begin{proof} Let $(X,\mathcal U_X),  (Y, \mathcal U_Y)$ be two objects in $Cord(\mathbb C)$. Let  $F\in \mathcal Z(X\times Y) $ be a homomorphism in $ Cord(\mathbb C)$.  The  corresponding homomorphsim in category $\mathcal N_k\mathcal C$ is defined 
 is defined to be $F_\ast$
$$\sigma \to (P_Y)_\ast [F\wedge (\sigma\otimes Y)]$$
where $\sigma\in \mathcal N_k\mathcal C(X)$. 
It suffices to show $F_\ast$ satisfies two conditions:
\par
(a) $F_\ast$ maps a $\mathcal N_k$ leveled current to a $\mathcal N_k$ leveled current.\par
(b) The map satisfies the composition criterion, i.e.
if $X, Y, W$ are smooth projective varieties over $\mathbb C$, and $Z_1, Z_2$ algebraic cycles are finite correspondences between
$X, Y$ and $Y, W$, then

\begin{equation}
(Z_2\circ Z_1)_\ast
= (Z_2)_\ast\circ (Z_1) _\ast
\end{equation}
where $Z_2\circ Z_1$ is the composition of finite correspondences.\par

Proof of (a):  By the definition of 
the intersection of currents, we obtain
$$deg(F_\ast(\sigma))=deg(\sigma).$$

Let $A$ be an algebraic set containing $\sigma$ such that the level of $\sigma$ is
$$k=deg(\sigma)-2deg_{\mathbb C} (A).$$

For any algebraic set $A$,   since  $F$ is a finite correspondence, 
\begin{equation}
deg(F_\ast(A))=deg(A).
\end{equation}
The level of $F_\ast(\sigma)$ is 
$$deg(F_\ast(\sigma))-2deg(F_\ast(A))$$
which is equal to
$$k=deg(\sigma)-2deg _{\mathbb C}(A).$$

The proof for contravariant $F^\ast$ is similar. 
\par

Proof of (b):  This is Proposition 4.3.

\end{proof}

\bigskip

\bigskip

Let 
\begin{equation}\begin{array}{ccc}
\kappa: \mathcal C(X) &\rightarrow & \mathscr D'(X)\\
T &\rightarrow & dT.
\end{array}\end{equation}
Let $\mathcal B(X)=\kappa^{-1}(\mathcal C(X))$ be the sublinear space that at least contain singular chains and $C^\infty$ forms.  
By Leibniz rule, part (6) of Property 2.6, the intersection $[\cdot \wedge \cdot]$ send 
$\mathcal B(X)\times \mathcal B(X)$ to $\mathcal B(X)$. So we let 
$$\mathcal N_k\mathcal B(X)=\mathcal N_k\mathcal C(X)\cap \mathcal B(X).$$

\bigskip

\begin{proposition}
 Then
\par
(a) $\mathcal N_{\bullet}\mathcal B(X)$ forms a decreasing  filtration of complex of $\mathcal B(X)$\par
(b) its spectral sequence $E_{\bullet}$  converges to
the $\mathbb R$ coefficiented,  algebraically leveled  filtration defined as  \begin{equation}
\mathcal N_k(X)=\sum_{r=dim(X)-k}^{2dim(X)-k} N^r H^{2r+k}(X), k=0, \cdots
\end{equation} on the total real cohomology $H(X;\mathbb R)$
(See [8] for details on leveled filtration).
 
\medskip

\end{proposition}

\bigskip

\bigskip

\begin{proof}
(a). Let $T\in   \mathcal N_k\mathcal B(X)$.  If $dim(T)\in [0, k+1]\cup [2n-k-1, 2n]$, then by the definition
$T$ is $\mathcal N_{k+1}$ leveled. If $dim(T)\in (k+1, 2n-k-1)$, then $dim(T)\in (k, 2n-k)$.  If $dim(T)$  is not in above cases,
there is an algebraic set $A$ such that 
$$supp(T)\subset A, \quad and\ dim(A)\leq {dim(T)+k\over 2}.$$
This implies that
$$supp(T)\subset A, \quad and\ dim(A)\leq {dim(T)+k+1\over 2}.$$
So $$T\in   \mathcal N_{k+1}\mathcal B(X).$$
This shows that 
$$ \mathcal N_0\mathcal B(X)\subset \cdots  \mathcal N_k\mathcal B(X)\subset  \mathcal N_{k+1}\mathcal B(X)\subset \cdots\subset \mathcal B(X)$$ form a
filtration.  
Since the differential $d$ on the differential form preserves the support, 
$d$ maps $ \mathcal N_k\mathcal B(X)$ to $ \mathcal N_k\mathcal B(X)$. Thus $( \mathcal N_{\bullet}\mathcal B, d)$
 is a decreasing filtration of the complex.

\par
(b) Because $\mathcal B(X)$ includes singular chains,  the limit of the spectral sequence,
$$\sum_{p, q}Gr^p (H^{p+ q}( \mathcal N_{\bullet}\mathcal B(X))) $$
 is a  filtration on the cohomology $$H(X;\mathbb R).$$
Notice that by the definition, this filtration is
the algebraically leveled filtration.

\end{proof}

\subsection{ Family of currents}

\quad\par

\begin{definition}\quad\par
 Let $S$ and $\mathcal X$ be manifolds equipped with De Rham data. 
Let $S\times \mathcal X$ be equipped with the product   De Rham data. 
Let $\mathcal I\in \mathcal C(S\times X)$ be a homogeneous Lebesgue current.  Let $P_{\mathcal X}$ be the projection
$$S\times \mathcal X\to \mathcal X.$$
We denote
\begin{equation}
(P_\mathcal X)_\ast [(\{s\}\otimes \mathcal X)\wedge \mathcal I ]
\end{equation}
by $\mathcal I_s$.  The set $\{\mathcal I_s\}$ for all such  $s$ in  $S$ will be called a family of 
of currents parametrized by  $S$. \par

\end{definition}

{\bf Remark} We should note that we have abused the notation to use $P_{\mathcal X}$ for its restriction 
$P_\mathcal X|_{\bullet} $ to a subset.   The family $\mathcal I_s$ depends on extrinsic De Rham data.

\bigskip

\begin{proposition} Let $\mathcal X$ be a  manifold. Let $S$ be real one dimensional.
Let $I_\epsilon\subset S$ be diffeomorphic to a finite closed interval of $\mathbb R$ with two end points $0$ and $\epsilon>0$. Let
 $S$ be equipped with De Rham data, $S\times  \mathcal X$ be equipped with 
a product De Rham data and $\mathcal X$ be equipped with corresponding the projection De Rham data.   
Let $\mathcal J$ be a  Lebesgue current 
\begin{equation}
S\times  \mathcal X.
\end{equation}
Then for  a closed current $T\in \mathcal C(\mathcal X)$, there is a well-defined  current 
$$\mathcal J_{I_\epsilon} (T)$$ on $\mathbb R\times \mathcal X$ such that
\par

(1)
\begin{align}\begin{split}
 &[ \mathcal J_\epsilon\wedge T]-[ \mathcal J_0\wedge T]
\\& =(-1)^{kp } d ((P_\mathcal X)_\ast \mathcal J_{I_\epsilon}(T))+(-1)^{p}[(I_\epsilon\otimes T)\wedge d\mathcal J]
\end{split}\end{align}
where $k=deg(\mathcal J), p=deg(T)$, $P_{\mathcal X}:\mathbb R\times \mathcal X\to\mathcal X$ is the projection. 
\par
(2) \begin{equation}\begin{split}
 &(P_\mathcal X)_\ast \mathcal J_{I_\epsilon}(T)=
\epsilon (P_\mathcal X)_\ast \mathcal J_{I_1}(T))\\
& [(I_\epsilon\otimes T)\wedge d\mathcal J]=\epsilon[(I_1\otimes T)\wedge d\mathcal J],
\end{split}\end{equation}
 and $[ \mathcal J_\epsilon\wedge T]$ is continuous in $\epsilon$ ( in the topology of currents) \footnote{ The continuity for higher dimensional parameter space $S$ does not hold}.

\end{proposition}

\bigskip

\begin{proof}  
Let $\mathbb R\times \mathcal X$ be equipped with a product De Rham data. 
We defined 
\begin{equation}
\mathcal J_{I_\epsilon}(T) =[ \mathcal J \wedge (I_\epsilon\otimes T) ].
\end{equation}
 
\par

Then we calculate
\begin{align}\begin{split}
& d\mathcal J_{I_\epsilon}(T)\\
 &(\text { by Leibniz Rule}.)
\\ &= [d (I_\epsilon\otimes T) \wedge \mathcal J]+(-1)^{p+1}[(I_\epsilon\otimes T)\wedge d\mathcal J]\\
&=[(\{\epsilon\}\otimes T-\{0\}\otimes T)\wedge \mathcal J]+(-1)^{p+1}[(I_\epsilon\otimes T)\wedge d\mathcal J]\\
 &(\text { by associativity and commutativity,  Property 2.6}.)\\
&=(-1)^{kp}  [ ((\{\epsilon\}-\{0\})\otimes \mathcal X)\wedge  \mathcal J\wedge (\mathbb R\otimes T)]+(-1)^{p+1}[(I_\epsilon\otimes T)\wedge d\mathcal J]
\end{split}\end{align}

By Definition 4.11, 
$$(P_\mathcal X)_\ast  [ \bigl((\{\epsilon\}-\{0\})\otimes \mathcal X\bigr)\wedge  \mathcal J]=[\mathcal J_\epsilon]-[\mathcal J_0].$$
Thus applying the projection formula, Proposition 2.8, we obtain
\begin{align}\begin{split}
(P_\mathcal X)_\ast   [ \bigl((\{\epsilon\}-\{0\})\otimes \mathcal X\bigr)\wedge  \mathcal J\wedge (\mathbb R\otimes T)]
=[ \mathcal J_\epsilon\wedge T]-[ \mathcal J_0\wedge T].
\end{split}\end{align}

Combination of (4.34), (4.35) is the assertion (4.31). 
We complete the proof. 

(2) We apply the construction (4.33). 
Then (4.32) follows from the equality of currents, 
$$ I_\epsilon=\epsilon I_1.$$
Finally by part (1) and formula (4.32) the family $ [ \mathcal J_\epsilon\wedge T]$ is continuous in $\epsilon$. 
\end{proof}

\bigskip

 \begin{ex}
Let $X$ be a smooth projective variety of dimension $n$ over $\mathbb C$.
Let $T$ be a closed Lebesgue current representing a non-zero primitive 
cohomology class in $H^n(X;\mathbb Q)$.
Let \begin{equation}
V\subset \mathbf P^1\times X.
\end{equation} be  a Lefschetz pencil in $X$.
Assume $\mathbf P^1\times X$ is equipped with a product De Rham data. 
Let \begin{equation}
\mathcal I=[V\wedge (\mathbf P^1\times T)].
\end{equation} be the intersection current. 
Then its fibre currents $\mathcal I_t$ are exact  for all $t$. But $\mathcal I$ is not  because $T$ is not. 
\end{ex}

\bigskip

\begin{ex}
Let $\mathcal X$ be a manifold and $T$  a non-zero homogeneous Lebesgue current in $\mathcal X$.
Let $S^1\times \mathcal X$ be equipped with a product De Rham data. Let $t_0$ be a point of 
$S^1$. Then $\mathcal I=\{t_0\}\otimes T$ gives a family of currents by Definition 4.11.
Notice $\mathcal I_t=0$ for all $t$ including $t_0$, but $\mathcal I$ is non-zero, and not closed provided $T$ is not.

\end{ex}

\bigskip

\section{Generalized Hodge conjecture on 3-folds}

 \bigskip

We give an application in complex geometry. 
\bigskip

\begin{theorem}

Generalized Hodge conjecture is correct on a 3-fold $X$.

\end{theorem}

\bigskip

\begin{proof} Let $X$ be a smooth projective variety over $\mathbb C$. In the cohomology vector space with rational coefficients, we
denote the coniveau filtration of coniveau $i$ and degree $2i+k$ 
by $$N^i H^{2i+k}(X)$$
(defined in [5] where the coniveau filtration is called the arithmetic filtration) and the linear span
of sub-Hodge structures of Hodge coniveau $i$ and degree $2i+k$ by 
$$M^iH^{2i+k}(X).$$
The generalized Hodge conjecture asserts that
\begin{equation}
M^iH^{2i+k}(X)=N^i H^{2i+k}(X).
\end{equation} for all existing $i, k$. For a 3-fold $X$, the only non-trivial cases are
\begin{align}
  & M^1H^{2}(X)=N^1 H^2(X) ? \\
 & M^1 H^3(X)= N^1 H^3(X) ?
\end{align}

The formula (5.2) is well-known as Lefschetz $(1, 1)$ theorem. For (5.3),  Delign's lemma 8.2.8, [1] already implies  that
 $$N^1 H^3(X)\subset M^1 H^3(X). $$
   Thus
it is sufficient to prove 
\begin{equation}
M^1 H^3(X)\subset N^1 H^3(X).
\end{equation}

Let $L\subset H^3(X; \mathbb Q)$ be a sub-Hodge structure of coniveau $1$.  So it is polarized. 
In [7], Voisin  used the intermediate Jacobian to prove a cohomological result that says 
   there is a smooth projective curve $C$, and a Hodge cycle \begin{equation}
\Psi\in  Hdg^4(C\times X)
\end{equation} such that
\begin{equation} \Psi_\ast (H^1(C;\mathbb Q))=L.
\end{equation}
where $\Psi_\ast$ is defined as the  image 
\begin{equation}\langle P\rangle _\ast \biggl( \Psi\cup ((\bullet)\otimes 1)) \biggr),\end{equation}
of the cohomological  map  $\langle P\rangle $ 
with the projection $P: C\times X\to X$.  
Next we convert the cohomological expression (5.7) to current's expression as 
\begin{equation}  P _\ast \biggl[  T\wedge ( (\bullet)\otimes 1)) \biggr],\end{equation}
where $T$ is any current representing $\Psi$. 
(The formula (5.8) is the key turning point of the proof that  has no difference from (5.7) in cohomology. However it carries the information of the support which is
 absent in (5.7)).  Next we analyze the current $T$. Notice $\langle P\rangle_\ast$ is a
Hodge morphism, so $\langle P\rangle_\ast (\Psi)$ is a Hodge cycle in $X$. 
By the Lefschetz (1, 1) theorem  $\langle P\rangle _\ast (\Psi)$  is algebraic on $X$. So 
there is a singular cycle $T_{\Psi}$ on $C\times X$ representing the class $ \Psi$ such that the projection in currents satisfies 
\begin{equation}
P_\ast(T_{\Psi})=S+b W
\end{equation}
where $S$ is a current of integration over the algebraic cycle $S$, and $b W$ is an exact Lebesgue current of dimension $4$ in $X$.
(adjust $ \Psi$ so $S$ is non-zero). 
  Consider another current in $C\times X$
\begin{equation}
T:=T_{\Psi}-[e]\otimes b W
\end{equation}
denoted by $T$, where $[e]$ is a current of
evaluation at a point $e\in C$.  Note $T$ is Lebesgue.
By adjusting the singular chain  $W$ continuously, we can assume the projection of the support of $T$ satisfies
 \begin{equation}
P (supp(T))=supp ( P_\ast (T)),
\end{equation}
i.e. the projection of the support is the support of the projection. (See appendix for the proof).  Thus we have the projection of currents
\begin{equation}
P_\ast (T)=S.
\end{equation}

Let $\Theta$ be the collection of closed Lebesgue currents on $C$ representing the classes in $H^1(C;\mathbb Q)$.  
Now we apply the real intersection theory  to establish the correspondence of currents ( in section 4.1), 
\begin{equation}
T_\ast (\Theta) 
\end{equation}
defined as (5.8). 
It gives  a family of currents (parametrized by $\Theta$) supported on the support of the current
$$
P_\ast (T)=S,
$$
which is the integration  over an algebraic cycle $S$, i.e. the the family of currents are all supported on the algebraic set
$|S|$.    This is  a criterion for coniveau filtration in terms of currents, i.e.
 for  $\beta\in  T_\ast (\Theta)$, the cohomology class $\langle \beta\rangle $ of $\beta$ satisfies  
\begin{equation}\begin{array}{ccc}
\langle \beta\rangle \in ker\biggl (H^3(X;\mathbb Q) &\rightarrow &  H^3(X-|S|;\mathbb Q)\biggr)
\end{array}\end{equation}

By Proposition 4.5 and Voisin's assertion (5.6), 
the collection of cosmological classes of the currents in $T_\ast (\Theta)$ consists of  all classes 
 in $L$. 
 This shows $L\subset N^1H^3(X)$. We complete the proof.

\end{proof}

\bigskip

\begin{appendices}

\section{Support of the projection}

\par

 In this Appendix, we study the supports of 
cellular cycles in a Cartesian product .\bigskip

Let $\mathcal X$ be a compact manifold of dimension $n$. We use the following setting in algebraic topology.
A $p$-singular simplex $S$ consists of
three elements: a $p$ dimensional  polyhedron $\Delta^p$  in $\mathbb R^v$, an orientation of
$\mathbb R^v$, and a $C^\infty$ map $f$ of $\mathbb R^v$ to  $X$. 
A chain is a linear combination of singular simplexes.  The support $|S|$ of $S$ is the image of $S$ in $X$.
A point in $S$ is a point in $|S|$.

Let $\mathcal  Y$ be another compact manifold of dimension $m$.  Let
\begin{equation}\begin{array}{ccc}
\mathcal P: \mathcal Y\times \mathcal X &\rightarrow & \mathcal X
\end{array}\end{equation}
be the projection.

\bigskip

\begin{definition}
Let $\sigma$ be a  $C^\infty$ $p$-singular simplex of $\mathcal Y\times \mathcal X$. Let $a$ be an interior point of $\sigma$.
If $$\mathcal P^{-1}\circ \mathcal P(a)\cap \sigma$$ is a finite set, we say 
$\sigma$ is finite at $a$. If $\sigma$ is finite at all interior points of $\sigma$, we say
$\sigma$ is finite to $\mathcal X$.  The chain is finite if each simplex in the chain is finite.
\end{definition}

\bigskip

\begin{proposition}
For any $C^\infty $ $p$-singular simplex  $\sigma$ in the coordinates chart of $\mathcal Y\times \mathcal X$ with $p\leq dim(\mathcal X)$, 
there is barycentric subdiviosn (multiple times) of $\sigma$
\begin{equation}
 Sd (\sigma)=\sum_{finite \ i} C_i, \end{equation}
such that 
each simplex $C_i$ is homotopic to another simplex finite to $\mathcal X$ and the homotopy is
a constant on the $\partial Sd(\sigma)$

\end{proposition}

\begin{proof}
Let $$\mathbb R^m, \mathbb R^n, \mathbb R^m\times \mathbb R^n$$ be the coordinate's charts for 
$\mathcal Y, \mathcal X, \mathcal Y\times \mathcal X$ respectively such that
$p\leq n$.  We would like to show that there is a multi barycentric subdivision  
to  divide  $\sigma$ to  a chain $\sum_{i=0}^N \sigma_i$ (a sum of smaller regular cells $\sigma_i$) such that there are homotopy $\sigma'_i$ for each 
$\sigma_i$  that  is finite to $\mathbb R^n$, and boundary of $\sigma_i'$ is the same as that of $\sigma_i$. 
   We use a claim to construct such small simplex $\sigma_i$. 

\bigskip

\begin{ass}
Let  $g: \mathbb R^k\to \mathbb R^l$ be a $C^\infty$ map with $l\geq k$. Let $q\in \mathbb R^k$ be a point. Then 
 there is an open ball $B$ of $q$ and continuous map $g': \mathbb R^k\to \mathbb R^l$. 
 such that\par
1) $g$ is homotopically deformed to $g'$ such that at all points on $\partial B$ and \par \hspace{1.5 CC} the  
 boundary $D$ of the unit ball $g$ is  fixed under the homotopy,\par

2) $g'$ in $B\backslash D$  is  $C^\infty$ and finite to one to its image in $\mathbb R^l$.

\end{ass}

\begin{proof} of Claim  A.3:  Let $\theta_1, \theta_2$ be two analytic functions 
on $\mathbb R^k$ such that
$\theta_1=\epsilon, \theta_2=0$ define $\partial B$ and $D$ where $\epsilon$ is  the radius of $B$.   
We consider the homotopy

\begin{equation}
(1-t) g+t ( g+\theta_1\theta_2 h),\quad\quad  t\in [0, 1].
\end{equation}
where $h$ is some  $C^\infty$ function.
Thus $g$ is homotopic to 
\begin{equation}
 \biggl(g+\theta_1\theta_2 h\biggr)
\end{equation}
The determinant of a maximal minor of the differential $J$ of \begin{equation}
 \biggl(g+\theta_1\theta_2 h\biggr)
\end{equation}
is a polynomial in $\theta_1, \theta_2$ whose coefficients are $C^\infty$ functions of $h$. 
Thus for a small $\epsilon$, by choosing a suitable $h$, the determinant is non-zero for all points in 
$B$ with $\theta_2\neq 1$ and $\theta_1\neq \epsilon$, i.e. the differential $J$ has full rank. 
By mean value theorem 
$$
 \biggl(g+\theta_1\theta_2 h\biggr)
$$ is 1-to-1 to its image when restricted to $B\backslash D$.  

It satisfies  required conditions in Claim A.3.

\end{proof}

Now let  $f$ be the composition of $$(\Delta^p)'\to \mathcal Y\times \mathcal X\to \mathcal X$$
where $(\Delta^p)'$ is a neighborhood of $\Delta^p$ .
Next we cover $ \bar \Delta^p$ with finitely many balls $B_i, i=1, \cdots, l$ and the homotopy in Claim A.3  for each $B_i$.
Consider the first open set $B_1$. Applying Claim A.3, $f$ is homotopic to $f_1: (\Delta^p)'\to \mathcal X$ such that
the homotopy fix the map $f$ on $\partial B_1$ and $D$ and $f$ is homotopic to $f_1$ which is finite-to-one 
on $B_1$. Then we repeat the homotopy from $f_1$ to $f_2$, from $ f_2$ to $f_3$, $\cdots$, from $f_{l-1}$ to $ f_l$. Finally, we obtain a continuous map $f_l$ which is finite-to-one
in each $B_i$ and is equal to $f$ on $D$.  Let $C_i$ be the barycentric subdivisions obtained from the covering $B_i, i=1, \cdots, l$.
Then $f_l$ is homotopy  to $f$. We complete the proof.

\end{proof}
\bigskip

\begin{proposition}
For any cellular  cycle $S$ in  $Y\times X$,  of dimension $p<dim(X)$, $S$ is homopotic to a  cycle finite to $X$.
\end{proposition}

\bigskip

\begin{proof} Let \begin{equation}
S=\sum_i C_i
\end{equation}
and each cell $C_i$ satisfies Proposition A.2 with a homotopy $h_i$.
Then there is synchronized homotopy with the same parameter $t\in [0, 1]$ such that
$C_i$ is homotopic to another cell $C_i'$  1-to-1 to $\mathcal X$, but the boundary is fixed. 
Since the boundaries are not changed, this synchronized homotopy are glued together to yield a homotopy of the cycle $S$, 
\begin{equation}
S'=\sum_i C_i'.
\end{equation}

\end{proof}
\bigskip

\end{appendices}

 \textsc{Department of Mathematics, Rhode Island college,  Providence, 
   RI 02908}\par
  \text{E-mail address}:  \texttt{binwang64319@gmail.com}, 
\text{Fax}:   \texttt{1-401-456-4695}

\end{document}